\documentclass[10pt]{article}
\usepackage{palatino}
\usepackage{amsmath}
\usepackage{amsthm}
\usepackage{amssymb}
\usepackage{latexsym}
\usepackage{pstricks}
\usepackage{graphicx, pst-plot, pst-node, pst-text, pst-tree}
\usepackage{titlesec}
\usepackage[small,it]{caption}
\usepackage{color}
\definecolor{lightgray}{rgb}{0.8, 0.8, 0.8}
\definecolor{darkgray}{rgb}{0.5, 0.5, 0.5}

\newtheorem{theorem}{Theorem}[section]
\newtheorem{proposition}[theorem]{Proposition}
\newtheorem{lemma}[theorem]{Lemma}

\newcommand{\newclaimproof}[2]{
   \newenvironment{#1}[1][]%
   {%
      \begin{trivlist}%
         \item[\hspace{\labelsep}\textnormal{\textit{#2.%
            \def\op@@@arg{##1}%
            \ifx\op@@@arg\empty
            \fi
         }}]%
   }%
   {%
     \hfill $\diamond$
      \end{trivlist}%
   }%
}

\newcounter{todocounter}

\newcommand{\Av}{\operatorname{Av}}

\newcommand{\C}{\mathcal{C}}
\newcommand{\D}{\mathcal{D}}

\newcommand{\Grid}{\operatorname{Grid}}
\newcommand{\Geom}{\operatorname{Geom}}

\newcommand{\st}{\::\:}

\newcommand{\bij}{\varphi}

\newcommand{\zpm}{0/\mathord{\pm} 1}
\newcommand{\ga}{\textsf{a}}
\newcommand{\gb}{\textsf{b}}
\newcommand{\gc}{\textsf{c}}
\newcommand{\gd}{\textsf{d}}
\newcommand{\xa}{x_\ga}
\newcommand{\xb}{x_\gb}
\newcommand{\xc}{x_\gc}
\newcommand{\xd}{x_\gd}
\newcommand{\OEISlink}[1]{#1}
\newcommand{\OEISref}{OEIS~\cite{sloane:the-on-line-enc:}}

\newcommand{\fnmatrix}[2]{\mbox{\begin{footnotesize}$\left(\begin{array}{#1}#2\end{array}\right)$\end{footnotesize}}}
\makeatletter
\newfont{\footsc}{cmcsc10 at 8truept}
\newfont{\footbf}{cmbx10 at 8truept}
\newfont{\footrm}{cmr10 at 10truept}
\renewenvironment{abstract}%
                {
                  \begin{list}{}%
                     {\setlength{\rightmargin}{1in}%
                      \setlength{\leftmargin}{1in}}%
                   \item[]\ignorespaces\begin{small}}%
                 {\end{small}\unskip\end{list}}

\newpagestyle{main}[\small]{
        \headrule
        \sethead[\usepage][][]
        {\sc Inflations of Geometric Grid Classes: Three Case Studies}{}{\usepage}}

\title{\sc Inflations of Geometric Grid Classes: Three Case Studies}
\author{%
Michael H. Albert\\[-0.25ex]
\small Department of Computer Science\\[-0.5ex]
\small University of Otago\\[-0.5ex]
\small Dunedin, New Zealand\\[1.5ex]
M. D. Atkinson\\[-0.25ex]
\small Department of Computer Science\\[-0.5ex]
\small University of Otago\\[-0.5ex]
\small Dunedin, New Zealand\\[1.5ex]
Vincent Vatter\footnote{Vatter's research was sponsored by the National Security Agency under Grant Number H98230-12-1-0207.  The United States Government is authorized to reproduce and distribute reprints not-withstanding any copyright notation herein.}\\[-0.25ex]
\small Department of Mathematics\\[-0.5ex]
\small University of Florida\\[-0.5ex]
\small Gainesville, Florida USA\\[-1.5ex]
}

\titleformat{\section}
        {\large\sc}
        {\thesection.}{1em}{}   

\date{}

\begin{document}
\maketitle

\pagestyle{main}

\begin{abstract}
We enumerate three specific permutation classes defined by two forbidden patterns of length four.   The techniques involve inflations of geometric grid classes.
\end{abstract}

\section{Introduction}\label{infinite-simples-intro}

\emph{Classes} of permutations are sets that are closed downwards under taking subpermutations.  They are often presented as sets $\C$ that avoid a given set $B$ of permutations (i.e. the members of $\C$ have no subpermutation in the set $B$).  We express this by the notation $\C=\Av(B)$.  We may take $B$ to be an antichain (a set of pairwise incomparable permutations), in which case we say that $B$ is the \emph{basis} of $\C$.

Much of the inspiration for the early work on permutation classes was driven by the enumeration problem: given $\C=\Av(B)$, how many permutations of each length does $\C$ contain?  The answer to such a question could be a formula giving this number $|\C_n|$ in terms of the length, $n$,  a generating function $\sum |\C_n|x^n$ or simply an asymptotic result about the behaviour of $|\C_n|$ as $n\rightarrow\infty$.

Recently, Albert, Atkinson, Bouvel, Ru\v{s}kuc, and Vatter~\cite{albert:geometric-grid-:} have developed the theory of geometric grid classes, and Albert, Ru\v{s}kuc, and Vatter~\cite{albert:inflations-of-g:} have continued this exploration by investigating the theory of inflations of such classes.  Our aim in this paper is to:
\begin{itemize}
\item demonstrate the effectiveness of this approach, and
\item illustrate how one might implement these techniques in a ``real world'' setting, bypassing what would otherwise be thorny theoretical issues.
\end{itemize}
It should be noted that this presentation is historically backward; the results of this paper preceded and inspired the two more theoretical papers cited above.

In this work, our examples are exclusively classes with two basis elements of length four, which we call \emph{2 $\times$ 4 classes}.  It must be admitted that the attention paid to 2 $\times$ 4 classes is not entirely in proportion to their intrinsic importance.  Nevertheless, these classes represent a significant dataset which seems to contain some difficult enumerative problems.  Thus they pose a good challenge for new approaches to the enumeration of restricted permutations.

There are $56$ essentially different (i.e. inequivalent under symmetries) 2 $\times$ 4 classes.  Some of these classes nevertheless share the same enumeration (a phenomenon called \emph{Wilf-equivalence}), so the 2 $\times$ 4 classes have only $38$ different enumerations \cite{bona:the-permutation:,kremer:permutations-wi:, kremer:postscript:-per:, kremer:finite-transiti:, le:wilf-classes-of:}.  This paper brings the number of $2\times 4$ Wilf classes which have been enumerated to $24$ (see Wikipedia~\cite{wikipedia:enumerations-of:}, which contains a list of such enumerations).

A central part of our approach depends on analysing the simple permutations in a class.  An \emph{interval} in the permutation $\pi$ is a set of contiguous indices $I=\{a,a+1,\dots,b\}$ such that the set $\{\pi(i)\st i\in I\}$ is also contiguous.  Every permutation $\pi$ of length $n$ has \emph{trivial intervals} of lengths $0$, $1$, and $n$, and other intervals are called \emph{proper}.  A permutation with no proper intervals is called \emph{simple}.  Another way to think about simple permutations arises repeatedly throughout our arguments.  Any subset  $p_1,\dots$ of entries of the permutation $\pi$ defines a minimal axes-parallel rectangle (or simply, \emph{box}), whose left edge slices through the leftmost of these entries, top edge slices through the greatest of these entries, and so on.  A simple permutation is one in which the box defined by any proper subset of two or more of its entries is \emph{separated} by an entry outside the box, by which we mean that this entry lies either
\begin{itemize}
\item vertically amongst these entries but to the left (or right) of all of them (\emph{vertical separation}), or
\item horizontally amongst these entries but above (or below) all of them (\emph{horizontal separation}).
\end{itemize}

Simple permutations are precisely those that do not arise from a non-trivial inflation, in the following sense.  Given a permutation $\sigma$ of length $m$ and nonempty permutations $\alpha_1,\dots,\alpha_m$, the \emph{inflation} of $\sigma$ by $\alpha_1,\dots,\alpha_m$,  denoted $\sigma[\alpha_1,\dots,\alpha_m]$, is the permutation of length $|\alpha_1|+\cdots+|\alpha_m|$ obtained by replacing each entry $\sigma(i)$ by an interval that is order isomorphic to $\alpha_i$ in such a way that the intervals are order isomorphic to $\sigma$.  For example,
\[
2413[1,132,321,12]=4\ 798\ 321\ 56. 
\]
We give two particular types of inflations special terminology and notation.  The inflation $12[\alpha_1,\alpha_2]$ is called a \emph{(direct) sum} and denoted by $\alpha_1\oplus\alpha_2$.  A permutation is \emph{sum decomposable} if it can be expressed as a nontrivial sum, and \emph{sum indecomposable} otherwise.  The inflation $21[\alpha_1,\alpha_2]$ is called a \emph{skew sum}, similarly denoted $\alpha_1\ominus\alpha_2$, and accompanied by analogous terms \emph{skew decomposable} and \emph{skew indecomposable}.  We extend the notion of direct and skew sum to classes, defining
\[
\C\oplus\D=\{\pi\oplus\sigma\st \pi\in\C\mbox{ and }\sigma\in\D\},
\]
with an analogous definition for $\C\ominus\D$.

The precise connection between simple permutations and inflations is furnished by the following result.

\begin{lemma}[Albert and Atkinson~\cite{albert:simple-permutat:}]
For every permutation $\pi$ there is a unique simple permutation $\sigma$ such that $\pi=\sigma[\alpha_1,\alpha_2,\ldots,\alpha_m]$.  Furthermore, except when $\sigma=12$ or $\sigma=21$, the intervals of $\pi$ that correspond to $\alpha_1,\alpha_2,\ldots,\alpha_m$ are uniquely determined.  In the case that $\sigma=12$ (respectively $\sigma=21$), the intervals are unique so long as we require the first of the two intervals to be sum (respectively skew) indecomposable.
\end{lemma}


One of the first general enumeration results is the following from \cite{albert:simple-permutat:}:

\begin{theorem}
\label{thm-fin-simples-alg}
If the class $\C$ contains only finitely many simple permutations, then $\C$ has an algebraic generating function.
\end{theorem}

This theorem has since been generalised in two different directions.  Brignall, Huczynska, and Vatter~\cite{brignall:simple-permutat:} introduced the notion of ``query-complete sets of properties'' to show that if a class satisfies the hypotheses of Theorem~\ref{thm-fin-simples-alg}, then such subsets as the even permutations or the involutions in $\C$ have algebraic generating functions.  More relevant to our investigation, 
\cite{albert:inflations-of-g:} significantly weakened the hypotheses of Theorem~\ref{thm-fin-simples-alg}, showing that its conclusion holds even when $\C$ contains infinitely many simple permutations, so long as these simple permutations lie in a geometric grid class, a notion introduced in Section~\ref{sec-geom-grid-review}.  Before this, we consider an example which gives the flavour of our approach without requiring much additional machinery.

\section{Example \#1: Avoiding 4213 and 3142}\label{sec-4213-3142}

Before describing our first example we need to introduce a family of simple permutations and quote a result.  A {\it parallel alternation\/} is a permutation whose plot can be divided into two parts, by a single horizontal or vertical line, so that the points on either side of this line are both either increasing or decreasing and for every pair of points from the same part there is a point from the other part which {\it separates\/} them, i.e., there is a point from the other part which lies either horizontally or vertically between them.  It is easy to see that a parallel alternation of length at least four is simple if and only if its length is even, it does not begin with its smallest entry, and it does not end with its greatest entry.  Thus there are precisely four simple parallel alternations of each even length at least six, shown in Figure~\ref{fig-par-alts}, and no simple parallel alternations of odd length.

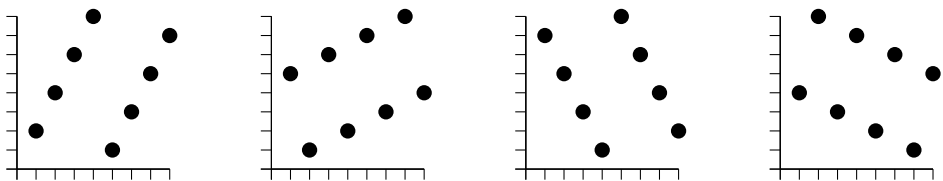
\begin{figure}
\begin{center}
\begin{tabular}{ccccccc}

\psset{xunit=0.01in, yunit=0.01in}
\psset{linewidth=0.005in}
\begin{pspicture}(0,0)(80,80)
\psaxes[dy=10,Dy=1,dx=10,Dx=1,tickstyle=bottom,showorigin=false,labels=none](0,0)(80,80)
\pscircle*(10,20){0.04in}
\pscircle*(20,40){0.04in}
\pscircle*(30,60){0.04in}
\pscircle*(40,80){0.04in}
\pscircle*(50,10){0.04in}
\pscircle*(60,30){0.04in}
\pscircle*(70,50){0.04in}
\pscircle*(80,70){0.04in}
\end{pspicture}

&\rule{0.2in}{0pt}&

\psset{xunit=0.01in, yunit=0.01in}
\psset{linewidth=0.005in}
\begin{pspicture}(0,0)(80,80)
\psaxes[dy=10,Dy=1,dx=10,Dx=1,tickstyle=bottom,showorigin=false,labels=none](0,0)(80,80)
\pscircle*(10,50){0.04in}
\pscircle*(20,10){0.04in}
\pscircle*(30,60){0.04in}
\pscircle*(40,20){0.04in}
\pscircle*(50,70){0.04in}
\pscircle*(60,30){0.04in}
\pscircle*(70,80){0.04in}
\pscircle*(80,40){0.04in}
\end{pspicture}

&\rule{0.2in}{0pt}&

\psset{xunit=0.01in, yunit=0.01in}
\psset{linewidth=0.005in}
\begin{pspicture}(0,0)(80,80)
\psaxes[dy=10,Dy=1,dx=10,Dx=1,tickstyle=bottom,showorigin=false,labels=none](0,0)(80,80)
\pscircle*(10,70){0.04in}
\pscircle*(20,50){0.04in}
\pscircle*(30,30){0.04in}
\pscircle*(40,10){0.04in}
\pscircle*(50,80){0.04in}
\pscircle*(60,60){0.04in}
\pscircle*(70,40){0.04in}
\pscircle*(80,20){0.04in}
\end{pspicture}

&\rule{0.2in}{0pt}&

\psset{xunit=0.01in, yunit=0.01in}
\psset{linewidth=0.005in}
\begin{pspicture}(0,0)(80,80)
\psaxes[dy=10,Dy=1,dx=10,Dx=1,tickstyle=bottom,showorigin=false,labels=none](0,0)(80,80)
\pscircle*(10,40){0.04in}
\pscircle*(20,80){0.04in}
\pscircle*(30,30){0.04in}
\pscircle*(40,70){0.04in}
\pscircle*(50,20){0.04in}
\pscircle*(60,60){0.04in}
\pscircle*(70,10){0.04in}
\pscircle*(80,50){0.04in}
\end{pspicture}

\end{tabular}
\end{center}
\caption{The four orientations of parallel alternations.}
\label{fig-par-alts}
\end{figure}

Schmerl and Trotter~\cite[Corollary 5.10]{schmerl:critically-inde:} proved a result (in the more general context of irreflexive binary relational structures) which in our context states that every simple permutation of length $n$ which is not a parallel alternation contains simple subpermutations of every length $5\le m\le n$.  Therefore, in order to establish that the permutation class $\C$ contains \emph{only} parallel alternations, we just need to check that it does not contain any simple permutation of length $5$, i.e., that
\[
\C\subseteq\Av(24153, 25314, 31524, 35142, 41352, 42513).
\]
Clearly this holds for the class $\Av(4213,3142)$, because $4213$ is contained in $25314$ and $42513$ while $3142$ is contained in $24153$, $31524$, $35142$, and $41352$.  Moreover, it is easily seen (because of the basis element $3142$) that $\Av(4213,3142)$ can contain only parallel alternations oriented as on the left of Figure~\ref{fig-par-alts}, i.e., those of the form
\[
246\cdots(2m)135\cdots(2m-1).
\]

With the simple permutations characterised, we now describe the allowed inflations.  It is easy to see that $\pi\oplus\sigma\in\Av(4213,3142)$ for all $\pi,\sigma\in\Av(4213,3142)$, or in other words, that the class is {\it sum closed\/}.  Thus, letting $f$ denote the generating function for nonempty permutations in $\Av(4213,3142)$ and $f_\oplus$ denote the generating function for sum decomposable permutations, we see that $f_\oplus=\left(f-f_\oplus\right)f$, from which it follows that
\[
f_\oplus=\frac{f^2}{1+f}.
\]
For skew sums, we have that $\pi\ominus\sigma\in\Av(4213,3142)$ if and only if $\pi\in\Av(4213,3142)$ and $\sigma\in\Av(213)$.  Letting
\[
c=\frac{1-2x-\sqrt{1-4x}}{2x}
\]
denote the generating function for the Catalan numbers (with constant term zero), which is well-known as the generating function of nonempty permutations in $\Av(213)$, we have $f_\ominus=\left(f-f_\ominus\right)c$, so
\[
f_\ominus=\frac{cf}{1+c}.
\]

Now we must count the inflations of the parallel alternations $246\cdots (2m)135\cdots (2m-1)$ for $m\ge 2$.  This is relatively straightforward:
\begin{itemize}
\item the interval inflating $2m-1$ must avoid $213$,
\item all other intervals inflating odd entries must be increasing, and
\item even entries may be inflated by any element of $\Av(4213,3142)$.
\end{itemize}
This shows that the contribution of inflations of $246\cdots (2m)135\cdots (2m-1)$, for each $m\ge 2$, is
\[
f^m\left(\frac{x}{1-x}\right)^{m-1}c,
\]
showing that
\[
f
=
x+f_\oplus+f_\ominus+\sum_{m=2}^\infty f^m\left(\frac{x}{1-x}\right)^{m-1}c
=
x+\frac{f^2}{1+f}+\frac{cf}{1+c}+\frac{xcf^2}{1-x-xf}.
\]
From this we obtain:

\begin{theorem}\label{thm-3142-4213}
The generating function $f$ for $\Av(4213,3142)$ satisfies
\[
\begin{array}{rclcc}
x^3f^6
&+&
(7x^3-7x^2+2x)f^5&&\\
&+&
(x^4+14x^3-21x^2+10x-1)f^4&&\\
&+&
(4x^4+8x^3-19x^2+11x-2)f^3&&\\
&+&
(6x^4-5x^3-2x^2+2x)f^2&&\\
&+&
(4x^4-7x^3+4x^2-x)f&&\\
&+&
x^4-2x^3+x^2
&=&0.
\end{array}
\]
\end{theorem}

The first several terms of this sequence are
\[
1,2,6,22,89,379,1664,7460,33977,156727,730619,3436710,16291842,77758962,
\]
sequence \OEISlink{A165541} in the \OEISref. The discriminant of the polynomial defining the generating function has a smallest positive root $\rho \approx 0.1895$, which is therefore the radius of convergence of the generating function and as $\Av(4213,3142)$ is sum closed (and hence the sequence $f_n$ is supermultiplicative) we can conclude that $f_n^{1/n} \to 1/\rho \approx 5.2778$. More detailed on the asymptotic behaviour of $f_n$ could be determined by standard methods as found for instance in Flajolet and Sedgewick~\cite[Section VII.7]{flajolet:analytic-combin:}.

\begin{figure}
\begin{center}
\begin{tabular}{ccccccc}

\psset{xunit=0.01in, yunit=0.01in}
\psset{linewidth=0.005in}
\begin{pspicture}(0,0)(80,80)
\psaxes[dy=10,Dy=1,dx=10,Dx=1,tickstyle=bottom,showorigin=false,labels=none](0,0)(80,80)
\pscircle*(10,20){0.04in}
\pscircle*(20,40){0.04in}
\pscircle*(30,60){0.04in}
\pscircle*(40,80){0.04in}
\pscircle*(50,10){0.04in}
\pscircle*(60,70){0.04in}
\pscircle*(70,50){0.04in}
\pscircle*(80,30){0.04in}
\end{pspicture}

&\rule{0.2in}{0pt}&

\psset{xunit=0.01in, yunit=0.01in}
\psset{linewidth=0.005in}
\begin{pspicture}(0,0)(80,80)
\psaxes[dy=10,Dy=1,dx=10,Dx=1,tickstyle=bottom,showorigin=false,labels=none](0,0)(80,80)
\pscircle*(10,30){0.04in}
\pscircle*(20,50){0.04in}
\pscircle*(30,70){0.04in}
\pscircle*(40,10){0.04in}
\pscircle*(50,80){0.04in}
\pscircle*(60,60){0.04in}
\pscircle*(70,40){0.04in}
\pscircle*(80,20){0.04in}
\end{pspicture}

&\rule{0.2in}{0pt}&

\psset{xunit=0.01in, yunit=0.01in}
\psset{linewidth=0.005in}
\begin{pspicture}(0,0)(80,80)
\psaxes[dy=10,Dy=1,dx=10,Dx=1,tickstyle=bottom,showorigin=false,labels=none](0,0)(80,80)
\pscircle*(10,40){0.04in}
\pscircle*(20,80){0.04in}
\pscircle*(30,10){0.04in}
\pscircle*(40,70){0.04in}
\pscircle*(50,20){0.04in}
\pscircle*(60,60){0.04in}
\pscircle*(70,30){0.04in}
\pscircle*(80,50){0.04in}
\end{pspicture}

&\rule{0.2in}{0pt}&

\psset{xunit=0.01in, yunit=0.01in}
\psset{linewidth=0.005in}
\begin{pspicture}(0,0)(80,80)
\psaxes[dy=10,Dy=1,dx=10,Dx=1,tickstyle=bottom,showorigin=false,labels=none](0,0)(80,80)
\pscircle*(10,50){0.04in}
\pscircle*(20,10){0.04in}
\pscircle*(30,80){0.04in}
\pscircle*(40,20){0.04in}
\pscircle*(50,70){0.04in}
\pscircle*(60,30){0.04in}
\pscircle*(70,60){0.04in}
\pscircle*(80,40){0.04in}
\end{pspicture}

\end{tabular}
\end{center}
\caption{Examples of {\it wedge simple permutations\/}.}\label{fig-wedge-simples}
\end{figure}
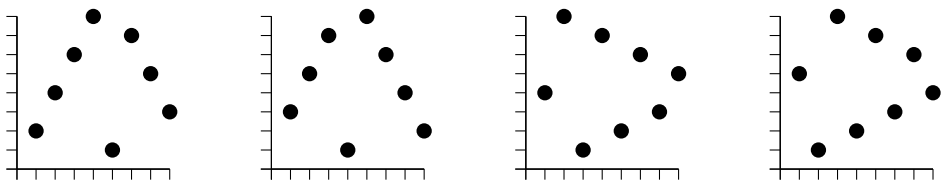

This is not the only 2 $\times$ 4 class to which such elementary techniques apply.  For example:
\begin{itemize}
\item $\Av(4213,1342)$ contains precisely two simple permutations of each length $n\ge 4$, both of which are wedge simple permutations oriented as the first two permutations shown in Figure~\ref{fig-wedge-simples}.  This family of simple permutations is well enough behaved that we could enumerate the class, but this has already been done by Kremer and Shiu~\cite{kremer:finite-transiti:} and can now be performed automatically using the Maple package {\sc FinLabel} described in Vatter~\cite{vatter:finitely-labele:}.
\item $\Av(4213,3124)$ contains precisely two simple permutations of each length $n\ge 4$, oriented as the rightmost two permutations shown in Figure~\ref{fig-wedge-simples}.  This class was enumerated by B\'ona~\cite{bona:the-permutation:}.
\end{itemize}

\section{Grid Classes and Regular Languages}\label{sec-geom-grid-review}

Given a permutation $\pi$ of length $n$ and sets $X,Y\subseteq[n]$, we write $\pi(X\times Y)$ for the permutation that is order isomorphic to the subsequence of $\pi$ with indices from $X$ and values in $Y$.  For example, $286435179([4,9]\times[5,9])$ consists of the subsequence of entries in indices $4$ through $9$ which have values between $5$ and $9$; in this case the subsequence is $579$, so $286435179([4,9]\times[5,9])=123$.

Suppose that $M$ is a $t\times u$ matrix%
\footnote{Note that in order for the cells of the matrix $M$ to be compatible with plots of permutations, we use Cartesian coordinates for our matrices, indexing them first by column, from left to right starting with $1$, and then by row, from bottom to top.}%
with entries from $\{0,\pm 1\}$.  A \emph{gridded permutation} is a permutation $\pi$ equipped with \emph{row} and \emph{column divisions} denoted respectively by $1=c_1\le\cdots\le c_{t+1}=n+1$ and $1=r_1\le\cdots\le r_{u+1}=n+1$ (where $n$ is the length of $\pi$).  This gridded permutation (or simply, gridding of $\pi$) is \emph{compatible} with the matrix $M$ (in which case we sometimes call it an $M$-gridding of $\pi$) if $\pi([c_k,c_{k+1})\times[r_\ell,r_{\ell+1}))$ is increasing whenever $M_{k,\ell}=1$, decreasing whenever $M_{k,\ell}=-1$, and empty whenever $M_{k,\ell}=0$.  The {\it (monotone) grid class of $M$\/}, written $\Grid(M)$, consists of all permutations which possess a gridding compatible with $M$.  Figure~\ref{fig-grid-geom-three-pane} shows an example.

As illustrated by Murphy and Vatter~\cite{murphy:profile-classes:}, monotone grid classes can display chaotic and unstructured behaviour.  However, it has recently been shown that these classes contain subclasses with especially amenable structure.  To define these subclasses, consider the point set in $\mathbb{R}^2$ (called the \emph{standard figure} of the $\zpm$ matrix $M$) consisting of cells $C_{kl}$ whose contents are:
\begin{itemize}
\item the line segment from $(k-1,\ell-1)$ to $(k,\ell)$ if $M_{k,\ell}=1$ or
\item the line segment from $(k-1,\ell)$ to $(k,\ell-1)$ if $M_{k,\ell}=-1$ or
\item empty if $M_{kl}=0$.
\end{itemize}
The \emph{geometric grid class} of $M$, denoted by $\Geom(M)$, is then the set of all permutations that can be drawn on this figure in the following manner.  Choose $n$ points in the figure, no two on a common horizontal or vertical line.  Then label the points from $1$ to $n$ from bottom to top and record these labels reading left to right.  The centre pane of Figure~\ref{fig-grid-geom-three-pane} shows a permutation from a geometric grid class, while the right pane demonstrates that $2413$ is not in this geometric grid class.

It sometimes happens that $\Grid(M)=\Geom(M)$; to characterise this phenomenon, we need to introduce a graph.  The \emph{row-column graph} of a $t\times u$ matrix $M$ is the bipartite graph on the vertices $x_1$, $\dots$, $x_t$, $y_1$, $\dots$, $y_u$ where $x_k$ is adjacent to $y_\ell$ if and only if $M_{k,\ell}\neq 0$.  Albert, Atkinson, Bouvel, Ru\v{s}kuc, and Vatter~{\cite[Theorem 3.2]{albert:geometric-grid-:}} showed that $\Grid(M)=\Geom(M)$ if and only if the row-column graph of $M$ is a forest (in this case we say that $M$ \emph{is} a forest).  As it happens, all gridding matrices encountered in this paper are forests.

Geometric grid classes are especially tractable because their elements can be encoded by words over a finite alphabet, and for the rest of this section we describe this encoding and its properties.  We say that a $\zpm$ matrix $M$ of size $t\times u$ is a \emph{partial multiplication matrix} if there exist \emph{column and row signs}
\[
f_1,\ldots,f_t,g_1,\ldots,g_u\in \{1,-1\}
\]
such that every entry $M_{k,\ell}$ is equal to either $f_kg_\ell$ or $0$.  It is not hard to prove that every geometric grid class is equal to $\Geom(M)$ for a partial multiplication matrix $M$, and this is especially trivial for forests.

\begin{figure}
\begin{center}
\begin{tabular}{ccccc}

\psset{xunit=0.015in, yunit=0.015in}
\psset{linewidth=0.005in}
\begin{pspicture}(0,0)(100,100)
\psline[linecolor=darkgray,linestyle=solid,linewidth=0.02in]{c-c}(35,0)(35,100)
\psline[linecolor=darkgray,linestyle=solid,linewidth=0.02in]{c-c}(0,45)(100,45)
\psaxes[dy=10,Dy=1,dx=10,Dx=1,tickstyle=bottom,showorigin=false,labels=none](0,0)(98,98)
\pscircle*(10,20){0.04in}
\pscircle*(20,80){0.04in}
\pscircle*(30,60){0.04in}
\pscircle*(40,40){0.04in}
\pscircle*(50,30){0.04in}
\pscircle*(60,50){0.04in}
\pscircle*(70,10){0.04in}
\pscircle*(80,70){0.04in}
\pscircle*(90,90){0.04in}
\end{pspicture}

&\rule{0.2in}{0pt}&

\psset{xunit=0.015in, yunit=0.015in}
\psset{linewidth=0.005in}
\begin{pspicture}(0,0)(100,100)
\psline[linecolor=black,linestyle=solid,linewidth=0.02in](0,0)(100,100)
\psline[linecolor=black,linestyle=solid,linewidth=0.02in](0,100)(100,0)
\psline[linecolor=darkgray,linestyle=solid,linewidth=0.02in]{c-c}(50,0)(50,100)
\psline[linecolor=darkgray,linestyle=solid,linewidth=0.02in]{c-c}(0,50)(100,50)
\psline[linecolor=darkgray,linestyle=solid,linewidth=0.02in]{c-c}(0,0)(100,0)(100,100)(0,100)(0,0)
\pscircle*(10,10){0.04in}
\pscircle*(20,80){0.04in}
\pscircle*(30,70){0.04in}
\pscircle*(40,40){0.04in}
\pscircle*(60,60){0.04in}
\pscircle*(70,30){0.04in}
\pscircle*(80,20){0.04in}
\pscircle*(90,90){0.04in}
\end{pspicture}

&\rule{0.2in}{0pt}&

\psset{xunit=0.015in, yunit=0.015in}
\psset{linewidth=0.005in}
\begin{pspicture}(0,0)(100,100)
\psline[linecolor=black,linestyle=dashed,linewidth=0.01in](20,20)(100,20)
\psline[linecolor=black,linestyle=dashed,linewidth=0.01in](20,20)(20,100)
\psline[linecolor=black,linestyle=dashed,linewidth=0.01in](25,75)(100,75)
\psline[linecolor=black,linestyle=dashed,linewidth=0.01in](70,70)(70,0)
\psline[linecolor=black,linestyle=solid,linewidth=0.02in](0,0)(100,100)
\psline[linecolor=black,linestyle=solid,linewidth=0.02in](0,100)(100,0)
\psline[linecolor=darkgray,linestyle=solid,linewidth=0.02in]{c-c}(50,0)(50,100)
\psline[linecolor=darkgray,linestyle=solid,linewidth=0.02in]{c-c}(0,50)(100,50)
\psline[linecolor=darkgray,linestyle=solid,linewidth=0.02in]{c-c}(0,0)(100,0)(100,100)(0,100)(0,0)
\pscircle*(20,20){0.04in}
\pscircle*(25,75){0.04in}
\pscircle*(70,70){0.04in}
\pscircle*(60,10){0.04in}
\end{pspicture}

\end{tabular}
\end{center}
\caption[]{The permutation $286435179$ lies in the grid class of the matrix
$$
M=\fnmatrix{rr}{-1&1\\1&-1},
$$
as the gridding on the left demonstrates.  The figure in the center shows that the permutation $17645328$ lies in $\Geom(M)$.  Finally, the figure on the right shows that $2413$ does not lie in the geometric grid class of $M$ (although it \emph{does} lie in $\Grid(M)$): traveling clockwise from  $2$, we see that  $4$ must lie closer to the centre than $2$,  $3$ must lie closer to the centre than  $4$, but then  $1$ must lie closer to the centre than  $3$ but further from  the centre than  $2$.}
\label{fig-grid-geom-three-pane}
\end{figure}
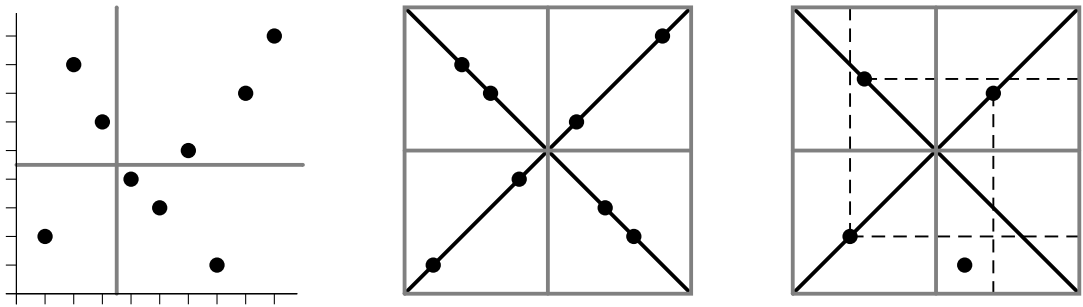

The column and row signs essentially specify an order in which the monotone entries in a cell of a gridded permutation should be read.  Cells corresponding to $M_{k\ell}=f_kg_{\ell}$ are read from left to right (respectively right to left) if $f_k=1$ (respectively $f_k=-1$) and bottom to top (respectively top to bottom) if $g_{\ell}=1$ (respectively $g_{\ell}=-1$).    These directions are sometimes marked on our diagrams.  The {\em base point} of a cell is the corner from which its reading begins.

To describe the encoding of $\Geom(M)$ we introduce a \emph{cell alphabet} $\Sigma$ associated to $M$ which consists of a unique letter $a_{kl}$ for each nonempty cell $C_{kl}$ of the standard figure of $M$.  Then, to every word $w=w_1\cdots w_n\in\Sigma^\ast$ we associate a permutation $\bij(w)$.  First we choose arbitrary distances
\[
0<d_1<\cdots<d_n<1.
\]
For each $1\le i\le n$, we choose a point $p_i$ corresponding to $w_i$ in the following manner.  Let $w_i=a_{k\ell}$; the point $p_i$ is chosen from the line segment in cell $C_{k,\ell}$, at infinity-norm distance $d_i$ from the base point of this cell.  Finally, $\bij(w)$ denotes the permutation defined by the set $\{p_1,\dots,p_n\}$ of points.  It can be seen that $\bij$ does not depend on the particular choice of $d_1,\dots,d_n$, and thus $\bij\st\Sigma^\ast\to \Geom(M)$ is a well-defined mapping.

The mapping $\bij$ is many-to-one, and so for enumerative applications we must restrict its domain to a set $L\subseteq\Sigma^\ast$ on which $\bij$ is injective.  We seek to choose $L$ to be a \emph{regular language}.  The regular languages are those that can be obtained from the empty language and the singleton languages using the operations of union, concatenation, and Kleene star (where $K^*$ is the set of all concatenations of 0 or more words from $K$).
%
Alternatively, regular languages can also be characterised as those accepted by deterministic finite state automata.  From this latter viewpoint it follows (e.g., by the transfer matrix method) that regular languages have rational generating functions (either when enumerated by length, or with a separate variable $x_a$ for each $a \in \Sigma$). We refer readers to~\cite[Section I.4 and Appendix A.7]{flajolet:analytic-combin:} for more information on regular languages. 


The following theorem from \cite{albert:geometric-grid-:} demonstrates the connection between subclasses of  geometric grid classes and regular languages. Essentially, it says that all such classes are extremely well behaved.

\begin{theorem}
\label{thm-geom-griddable-all}
Suppose that $\C\subseteq\Geom(M)$ is a permutation class and $M$ is a partial multiplication matrix with cell alphabet $\Sigma$.  Then the following hold:
\begin{enumerate}
\item[(i)] $\C$ is partially well-ordered.
\item[(ii)] $\C$ is finitely based.
\item[(iii)] There is a regular language $L\subseteq\Sigma^\ast$ such that the mapping $\bij\st L\rightarrow\C$ is a bijection.
\item[(iv)] There is a regular language $L_S$, contained in the regular language $L$ from (iii), such that the mapping $\bij$ is a bijection between $L_S$ and the simple permutations in $\C$.
\end{enumerate}
\end{theorem}

Note that the proof of Theorem~\ref{thm-geom-griddable-all} is nonconstructive, so while we use the encoding $\bij$ throughout this work, we construct the regular languages we use from first principles.

Albert, Atkinson, and Brignall~\cite{albert:the-enumeration:2143:4231, albert:the-enumeration:3} demonstrate four concrete examples of using these techniques to enumerate $2\times 4$ classes.

The examples considered in this paper are inflations of geometric grid classes.  The theoretical issues of such classes were studied by Albert, Ru\v{s}kuc, and Vatter~\cite{albert:inflations-of-g:}, who proved that every subclass of $\langle\Geom(M)\rangle$ has an algebraic generating function (essentially by showing that it is in bijection with a context-free language).  From this perspective, Section~\ref{sec-4213-3142} considers the case of a class contained in $\langle\Geom(\begin{footnotesize}\begin{array}{rr}1&1\end{array}\end{footnotesize})\rangle$, while the next two sections consider classes whose simple permutations are contained in more complicated geometric grid classes.

\section{Example \#2: Avoiding $4312$ and $3142$}\label{sec-4312-3142}

We begin our next example with a characterisation of its simple permutations.

\begin{proposition}\label{prop-4312-3142-simples}
The simple permutations of $\Av(4312,3142)$ and $\Geom\fnmatrix{rrr}{0&1&1\\1&0&-1}$ coincide.
\end{proposition}
\begin{proof}
First observe that
\[
\Grid\fnmatrix{rrr}{0&1&1\\1&0&-1}
=
\Geom\fnmatrix{rrr}{0&1&1\\1&0&-1}
\subseteq
\Av(4312, 3142),
\]
so it suffices to prove that the simple permutations in $\Av(4312, 3142)$ are contained this grid class. Specifically, we will show that in any simple permutation of $\Av(4312, 3142)$ the entries that follow the maximum make up a ``wedge permutation'' oriented as $<$  (which is equivalent to avoiding both 132 and 312), and those preceding the maximum form an increasing sequence.

So, let a simple permutation $\pi \in \Av(4312, 3142)$ be given. Because $\pi$ avoids 4312 there can be no 312 pattern after its maximum so, for the sake of contradiction, assume that there is a $132$ pattern. Specifically, choose such a pattern $acb$ where $a$ is as low as possible, and $c$ is as high as possible (for the chosen $a$). This yields the situation depicted on the left in Figure \ref{fig-ex2-simple-no-132-after-max}.  Now, in order that the cell bounded by $\{b,c\}$ not form an interval, there must be some entry $d$ in the cell immediately to its left.  Taking the leftmost such entry yields the diagram on the right in Figure \ref{fig-ex2-simple-no-132-after-max}. In this diagram we see that the entries of $\pi$ lying in the box bounded by $\{b,c,d\}$ (including those three entries) form a proper interval, contradicting the simplicity of $\pi$.

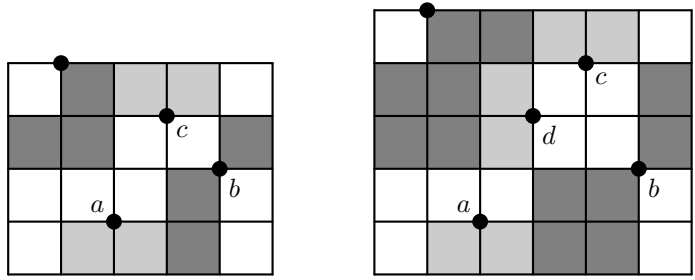
\begin{figure}
\begin{center}
\begin{tabular}{ccc}

\psset{xunit=2pt, yunit=2pt, runit=1.5pt}
\begin{pspicture}(0,0)(50,40)
\pspolygon*[linecolor=darkgray](0,20)(10,20)(10,30)(0,30)
\pspolygon*[linecolor=lightgray](10,0)(20,0)(20,10)(10,10)
\pspolygon*[linecolor=darkgray](10,20)(20,20)(20,30)(10,30)
\pspolygon*[linecolor=darkgray](10,30)(20,30)(20,40)(10,40)
\pspolygon*[linecolor=lightgray](20,0)(30,0)(30,10)(20,10)
\pspolygon*[linecolor=lightgray](20,30)(30,30)(30,40)(20,40)
\pspolygon*[linecolor=darkgray](30,0)(40,0)(40,10)(30,10)
\pspolygon*[linecolor=darkgray](30,10)(40,10)(40,20)(30,20)
\pspolygon*[linecolor=lightgray](30,30)(40,30)(40,40)(30,40)
\pspolygon*[linecolor=darkgray](40,20)(50,20)(50,30)(40,30)
\pscircle*[linecolor=black](10,40){2.0}
\pscircle*[linecolor=black](20,10){2.0}
\uput[135](20,10){$a$}
\pscircle*[linecolor=black](30,30){2.0}
\uput[-45](30,30){$c$}
\pscircle*[linecolor=black](40,20){2.0}
\uput[-45](40,20){$b$}
\psline{c-c}(0,0)(0,40)
\psline{c-c}(0,0)(50,0)
\psline{c-c}(10,0)(10,40)
\psline{c-c}(0,10)(50,10)
\psline{c-c}(20,0)(20,40)
\psline{c-c}(0,20)(50,20)
\psline{c-c}(30,0)(30,40)
\psline{c-c}(0,30)(50,30)
\psline{c-c}(40,0)(40,40)
\psline{c-c}(0,40)(50,40)
\psline{c-c}(50,0)(50,40)
\end{pspicture}

&\rule{0.2in}{0pt}&

\psset{xunit=2pt, yunit=2pt, runit=1.5pt}
\begin{pspicture}(0,0)(60,50)
\pspolygon*[linecolor=darkgray](0,20)(10,20)(10,30)(0,30)
\pspolygon*[linecolor=darkgray](0,30)(10,30)(10,40)(0,40)
\pspolygon*[linecolor=lightgray](10,0)(20,0)(20,10)(10,10)
\pspolygon*[linecolor=darkgray](10,20)(20,20)(20,30)(10,30)
\pspolygon*[linecolor=darkgray](10,30)(20,30)(20,40)(10,40)
\pspolygon*[linecolor=darkgray](10,40)(20,40)(20,50)(10,50)
\pspolygon*[linecolor=lightgray](20,0)(30,0)(30,10)(20,10)
\pspolygon*[linecolor=lightgray](20,20)(30,20)(30,30)(20,30)
\pspolygon*[linecolor=lightgray](20,30)(30,30)(30,40)(20,40)
\pspolygon*[linecolor=darkgray](20,40)(30,40)(30,50)(20,50)
\pspolygon*[linecolor=darkgray](30,0)(40,0)(40,10)(30,10)
\pspolygon*[linecolor=darkgray](30,10)(40,10)(40,20)(30,20)
\pspolygon*[linecolor=lightgray](30,40)(40,40)(40,50)(30,50)
\pspolygon*[linecolor=darkgray](40,0)(50,0)(50,10)(40,10)
\pspolygon*[linecolor=darkgray](40,10)(50,10)(50,20)(40,20)
\pspolygon*[linecolor=lightgray](40,40)(50,40)(50,50)(40,50)
\pspolygon*[linecolor=darkgray](50,20)(60,20)(60,30)(50,30)
\pspolygon*[linecolor=darkgray](50,30)(60,30)(60,40)(50,40)
\pscircle*[linecolor=black](10,50){2.0}
\pscircle*[linecolor=black](20,10){2.0}
\uput[135](20,10){$a$}
\pscircle*[linecolor=black](30,30){2.0}
\uput[-45](30,30){$d$}
\pscircle*[linecolor=black](40,40){2.0}
\uput[-45](40,40){$c$}
\pscircle*[linecolor=black](50,20){2.0}
\uput[-45](50,20){$b$}
\psline{c-c}(0,0)(0,50)
\psline{c-c}(0,0)(60,0)
\psline{c-c}(10,0)(10,50)
\psline{c-c}(0,10)(60,10)
\psline{c-c}(20,0)(20,50)
\psline{c-c}(0,20)(60,20)
\psline{c-c}(30,0)(30,50)
\psline{c-c}(0,30)(60,30)
\psline{c-c}(40,0)(40,50)
\psline{c-c}(0,40)(60,40)
\psline{c-c}(50,0)(50,50)
\psline{c-c}(0,50)(60,50)
\psline{c-c}(60,0)(60,50)
\end{pspicture}

\end{tabular}
\end{center}
\caption[]{The illustration that a simple permutation of $\Av(4312, 3142)$ can have no 132 pattern after its maximum.  Dark grey regions cannot be occupied due to the avoidance conditions, lighter regions because of choices made when the entries were selected (topmost, leftmost, etc.).}
\label{fig-ex2-simple-no-132-after-max}
\end{figure}

We can now argue in a similar fashion that the entries preceding the maximum entry of $\pi$ form an increasing sequence, i.e., that there cannot be a $21$ pattern before the maximum entry of $\pi$.  Suppose to the contrary that there were one, and choose such a pattern $ba$ where $b$ is as high as possible, and $a$ is as low as possible (for the chosen $b$). Now the cell defined by $\{a,b\}$ must be split either to the left or to the right.  The picture on the left of Figure~\ref{fig-ex2-simple-no-21-after-max} shows that $\{a,b\}$ cannot be split solely to the right, as then taking $c$ to be the rightmost such separator we see that $\{a,b,c\}$ would lie in a proper interval.  Similarly, $\{a,b\}$ cannot be split solely to the left.  Thus $\{a,b\}$ must be split on both the left and the right.  Now, taking $c$ to be the rightmost separator and $d$ the leftmost separator, we have two cases, depicted in the centre and right of Figure~\ref{fig-ex2-simple-no-21-after-max}.  In both cases it is clear that $\{a,b,c,d\}$ is contained in a proper interval, and this contradiction completes the proof.
\end{proof}

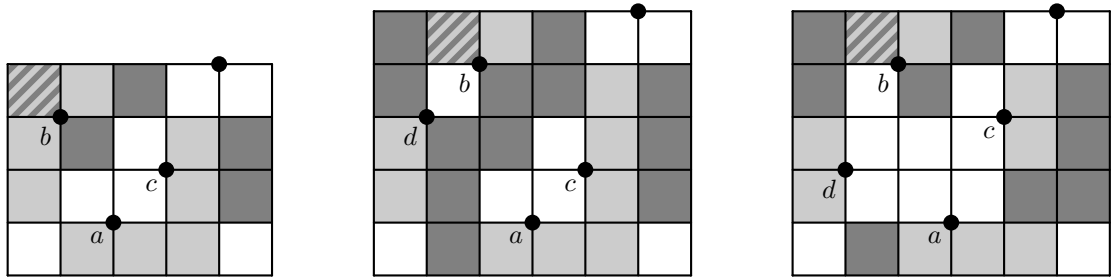
\begin{figure}
\begin{center}
\begin{tabular}{ccccc}

\psset{xunit=2pt, yunit=2pt, runit=1.5pt}
\begin{pspicture}(0,0)(50,40)
\pspolygon*[linecolor=lightgray](0,10)(10,10)(10,20)(0,20)
\pspolygon*[linecolor=lightgray](0,20)(10,20)(10,30)(0,30)
\pspolygon*[linecolor=lightgray](0,30)(10,30)(10,40)(0,40)
\psframe[linewidth=0,linecolor=white,fillstyle=hlines,hatchcolor=darkgray,hatchwidth=1.5,hatchsep=1.5](0,30)(10,40)
\pspolygon*[linecolor=lightgray](10,0)(20,0)(20,10)(10,10)
\pspolygon*[linecolor=darkgray](10,20)(20,20)(20,30)(10,30)
\pspolygon*[linecolor=lightgray](10,30)(20,30)(20,40)(10,40)
\pspolygon*[linecolor=lightgray](20,0)(30,0)(30,10)(20,10)
\pspolygon*[linecolor=darkgray](20,30)(30,30)(30,40)(20,40)
\pspolygon*[linecolor=lightgray](30,0)(40,0)(40,10)(30,10)
\pspolygon*[linecolor=lightgray](30,10)(40,10)(40,20)(30,20)
\pspolygon*[linecolor=lightgray](30,20)(40,20)(40,30)(30,30)
\pspolygon*[linecolor=darkgray](40,10)(50,10)(50,20)(40,20)
\pspolygon*[linecolor=darkgray](40,20)(50,20)(50,30)(40,30)
\pscircle*[linecolor=black](10,30){2.0}
\uput[-135](10,30){$b$}
\pscircle*[linecolor=black](20,10){2.0}
\uput[-135](20,10){$a$}
\pscircle*[linecolor=black](30,20){2.0}
\uput[-135](30,20){$c$}
\pscircle*[linecolor=black](40,40){2.0}
\psline{c-c}(0,0)(0,40)
\psline{c-c}(0,0)(50,0)
\psline{c-c}(10,0)(10,40)
\psline{c-c}(0,10)(50,10)
\psline{c-c}(20,0)(20,40)
\psline{c-c}(0,20)(50,20)
\psline{c-c}(30,0)(30,40)
\psline{c-c}(0,30)(50,30)
\psline{c-c}(40,0)(40,40)
\psline{c-c}(0,40)(50,40)
\psline{c-c}(50,0)(50,40)
\end{pspicture}

&\rule{0.2in}{0pt}&

\psset{xunit=2pt, yunit=2pt, runit=1.5pt}
\begin{pspicture}(0,0)(60,50)
\pspolygon*[linecolor=lightgray](0,10)(10,10)(10,20)(0,20)
\pspolygon*[linecolor=lightgray](0,20)(10,20)(10,30)(0,30)
\pspolygon*[linecolor=darkgray](0,30)(10,30)(10,40)(0,40)
\pspolygon*[linecolor=darkgray](0,40)(10,40)(10,50)(0,50)
\pspolygon*[linecolor=darkgray](10,0)(20,0)(20,10)(10,10)
\pspolygon*[linecolor=darkgray](10,10)(20,10)(20,20)(10,20)
\pspolygon*[linecolor=darkgray](10,20)(20,20)(20,30)(10,30)
\pspolygon*[linecolor=lightgray](10,40)(20,40)(20,50)(10,50)
\psframe[linewidth=0,linecolor=white,fillstyle=hlines,hatchcolor=darkgray,hatchwidth=1.5,hatchsep=1.5](10,40)(20,50)
\pspolygon*[linecolor=lightgray](20,0)(30,0)(30,10)(20,10)
\pspolygon*[linecolor=darkgray](20,20)(30,20)(30,30)(20,30)
\pspolygon*[linecolor=darkgray](20,30)(30,30)(30,40)(20,40)
\pspolygon*[linecolor=lightgray](20,40)(30,40)(30,50)(20,50)
\pspolygon*[linecolor=lightgray](30,0)(40,0)(40,10)(30,10)
\pspolygon*[linecolor=darkgray](30,30)(40,30)(40,40)(30,40)
\pspolygon*[linecolor=darkgray](30,40)(40,40)(40,50)(30,50)
\pspolygon*[linecolor=lightgray](40,0)(50,0)(50,10)(40,10)
\pspolygon*[linecolor=lightgray](40,10)(50,10)(50,20)(40,20)
\pspolygon*[linecolor=lightgray](40,20)(50,20)(50,30)(40,30)
\pspolygon*[linecolor=lightgray](40,30)(50,30)(50,40)(40,40)
\pspolygon*[linecolor=darkgray](50,10)(60,10)(60,20)(50,20)
\pspolygon*[linecolor=darkgray](50,20)(60,20)(60,30)(50,30)
\pspolygon*[linecolor=darkgray](50,30)(60,30)(60,40)(50,40)
\pscircle*[linecolor=black](10,30){2.0}
\uput[-135](10,30){$d$}
\pscircle*[linecolor=black](20,40){2.0}
\uput[-135](20,40){$b$}
\pscircle*[linecolor=black](30,10){2.0}
\uput[-135](30,10){$a$}
\pscircle*[linecolor=black](40,20){2.0}
\uput[-135](40,20){$c$}
\pscircle*[linecolor=black](50,50){2.0}
\psline{c-c}(0,0)(0,50)
\psline{c-c}(0,0)(60,0)
\psline{c-c}(10,0)(10,50)
\psline{c-c}(0,10)(60,10)
\psline{c-c}(20,0)(20,50)
\psline{c-c}(0,20)(60,20)
\psline{c-c}(30,0)(30,50)
\psline{c-c}(0,30)(60,30)
\psline{c-c}(40,0)(40,50)
\psline{c-c}(0,40)(60,40)
\psline{c-c}(50,0)(50,50)
\psline{c-c}(0,50)(60,50)
\psline{c-c}(60,0)(60,50)
\end{pspicture}

&\rule{0.2in}{0pt}&

\psset{xunit=2pt, yunit=2pt, runit=1.5pt}
\begin{pspicture}(0,0)(60,50)
\pspolygon*[linecolor=lightgray](0,10)(10,10)(10,20)(0,20)
\pspolygon*[linecolor=lightgray](0,20)(10,20)(10,30)(0,30)
\pspolygon*[linecolor=darkgray](0,30)(10,30)(10,40)(0,40)
\pspolygon*[linecolor=darkgray](0,40)(10,40)(10,50)(0,50)
\pspolygon*[linecolor=darkgray](10,0)(20,0)(20,10)(10,10)
\pspolygon*[linecolor=lightgray](10,40)(20,40)(20,50)(10,50)
\psframe[linewidth=0,linecolor=white,fillstyle=hlines,hatchcolor=darkgray,hatchwidth=1.5,hatchsep=1.5](10,40)(20,50)
\pspolygon*[linecolor=lightgray](20,0)(30,0)(30,10)(20,10)
\pspolygon*[linecolor=darkgray](20,30)(30,30)(30,40)(20,40)
\pspolygon*[linecolor=lightgray](20,40)(30,40)(30,50)(20,50)
\pspolygon*[linecolor=lightgray](30,0)(40,0)(40,10)(30,10)
\pspolygon*[linecolor=darkgray](30,40)(40,40)(40,50)(30,50)
\pspolygon*[linecolor=lightgray](40,0)(50,0)(50,10)(40,10)
\pspolygon*[linecolor=darkgray](40,10)(50,10)(50,20)(40,20)
\pspolygon*[linecolor=lightgray](40,20)(50,20)(50,30)(40,30)
\pspolygon*[linecolor=lightgray](40,30)(50,30)(50,40)(40,40)
\pspolygon*[linecolor=darkgray](50,10)(60,10)(60,20)(50,20)
\pspolygon*[linecolor=darkgray](50,20)(60,20)(60,30)(50,30)
\pspolygon*[linecolor=darkgray](50,30)(60,30)(60,40)(50,40)
\pscircle*[linecolor=black](10,20){2.0}
\uput[-135](10,20){$d$}
\pscircle*[linecolor=black](20,40){2.0}
\uput[-135](20,40){$b$}
\pscircle*[linecolor=black](30,10){2.0}
\uput[-135](30,10){$a$}
\pscircle*[linecolor=black](40,30){2.0}
\uput[-135](40,30){$c$}
\pscircle*[linecolor=black](50,50){2.0}
\psline{c-c}(0,0)(0,50)
\psline{c-c}(0,0)(60,0)
\psline{c-c}(10,0)(10,50)
\psline{c-c}(0,10)(60,10)
\psline{c-c}(20,0)(20,50)
\psline{c-c}(0,20)(60,20)
\psline{c-c}(30,0)(30,50)
\psline{c-c}(0,30)(60,30)
\psline{c-c}(40,0)(40,50)
\psline{c-c}(0,40)(60,40)
\psline{c-c}(50,0)(50,50)
\psline{c-c}(0,50)(60,50)
\psline{c-c}(60,0)(60,50)
\end{pspicture}
\end{tabular}
\end{center}
\caption[]{The illustration that a simple permutation of $\Av(4312, 3142)$ can have no 21 pattern before its maximum.  In these pictures, the hatched regions cannot be occupied both by the avoidance conditions and the choices made when the entries were selected.}
\label{fig-ex2-simple-no-21-after-max}
\end{figure}

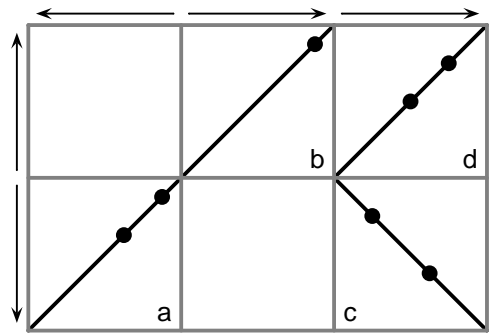
\begin{figure}
\begin{center}

\psset{xunit=0.02in, yunit=0.02in, runit=1.5pt}
\psset{linewidth=0.005in}
\begin{pspicture}(-3,0)(120,85)
\psline[linecolor=black,linestyle=solid,linewidth=0.02in]{c-c}(0,0)(80,80)
\psline[linecolor=black,linestyle=solid,linewidth=0.02in]{c-c}(80,40)(120,0)
\psline[linecolor=black,linestyle=solid,linewidth=0.02in]{c-c}(80,40)(120,80)
\psline[linecolor=darkgray,linestyle=solid,linewidth=0.02in]{c-c}(0,0)(0,80)
\psline[linecolor=darkgray,linestyle=solid,linewidth=0.02in]{c-c}(40,0)(40,80)
\psline[linecolor=darkgray,linestyle=solid,linewidth=0.02in]{c-c}(80,0)(80,80)
\psline[linecolor=darkgray,linestyle=solid,linewidth=0.02in]{c-c}(120,0)(120,80)
\psline[linecolor=darkgray,linestyle=solid,linewidth=0.02in]{c-c}(0,0)(120,0)
\psline[linecolor=darkgray,linestyle=solid,linewidth=0.02in]{c-c}(0,40)(120,40)
\psline[linecolor=darkgray,linestyle=solid,linewidth=0.02in]{c-c}(0,80)(120,80)
\psline[linecolor=black,linestyle=solid,linewidth=0.01in,arrowsize=0.06in]{<-c}(-3,2)(-3,38)
\psline[linecolor=black,linestyle=solid,linewidth=0.01in,arrowsize=0.06in]{c->}(-3,42)(-3,78)
\psline[linecolor=black,linestyle=solid,linewidth=0.01in,arrowsize=0.06in]{<-c}(2,83)(38,83)
\psline[linecolor=black,linestyle=solid,linewidth=0.01in,arrowsize=0.06in]{c->}(42,83)(78,83)
\psline[linecolor=black,linestyle=solid,linewidth=0.01in,arrowsize=0.06in]{c->}(82,83)(118,83)
\pscircle*(35,35){2.0}
\pscircle*(90,30){2.0}
\pscircle*(25,25){2.0}
\pscircle*(100,60){2.0}
\pscircle*(105,15){2.0}
\pscircle*(110,70){2.0}
\pscircle*(75,75){2.0}
\uput[135](40,0){$\ga$}
\uput[135](80,40){$\gb$}
\uput[45](80,0){$\gc$}
\uput[135](120,40){$\gd$}
\end{pspicture}

\end{center}
\caption{The word $\ga\gc\ga\gd\gc\gd\gb$ is mapped by $\bij$ to the simple permutation $2473516$.}\label{fig-2473516}
\end{figure}

We now consider the encoding $\bij$ over the cell alphabet $\Sigma=\{\ga,\gb,\gc,\gd\}$ as indicated in Figure~\ref{fig-2473516}, which also shows an example of $\bij$.  This mapping is not injective on $\Sigma^\ast$ for the following two reasons.
\begin{itemize}
\item[(G1)] The same gridded permutation may be the image of multiple words (in our example, this occurs because the pairs $\{\ga,\gb\}$, $\{\ga,\gd\}$, and $\{\gb,\gc\}$ ``commute'', i.e., they may be interchanged without affecting the gridded permutation obtained).  A method to handle this issue in general (by appealing to the theory of ``trace monoids'') is presented in \cite[Section 7]{albert:geometric-grid-:}.
\item[(G2)] A given permutation may have several different $M$-griddings.  A (nonconstructive) method to handle this issue in general is presented in \cite[Section 8]{albert:geometric-grid-:}.
\end{itemize}
In the class we are considering, $\Av(4312,3142)$, it is possible to deal with the issues concretely.

First we address (G1).  For any particular gridded permutation, we prefer the lexicographically minimal word encoding it.  For example, suppose that a word contained a factor of the form $\{\gb,\gd\}^+\ga$ (here the $+$ superscript signifies that this portion of the word contains at least one letter).  We could then replace this factor by a factor of the form $\ga\{\gb,\gd\}^+$ and obtain a lexicographically lesser word which is mapped to the same permutation.  Therefore we forbid factors of the form $\{\gb,\gd\}^+\ga$.  The other factor we need to forbid is $\gc\ga^\ast\gb$ (which could be replaced by a factor of the form $\gb\gc\ga^\ast$).

Now we address (G2), which requires us to choose a preferred (geometric) $M$-gridding for every permutation in $\Geom(M)$.  Among all $M$-griddings of a permutation, we prefer the one that has the most entries in the first column, then the most entries in the second column, and then the most entries in the first row.  Thus in terms of column divisions $1=c_1\le c_2\le c_3\le c_4=n+1$ and row divisions $1=r_1\le r_2\le r_3=n+1$, we seek to maximise $c_2$, then $c_3$, and then $r_2$.  The words which correspond to such griddings can now be characterised as those which do not begin with $\ga^\ast\gd$, $\gb$, $\{\ga,\gc\}^\ast\gb$, or $\gd$ and are not of the form $\gc\{\ga,\gc,\gd\}^\ast$.

With this language we may enumerate the grid class itself%
\footnote{This grid class (which, because it can be viewed as a ``juxtaposition'' in the sense of Atkinson~\cite{atkinson:restricted-perm:}, can be shown to have basis $\{2143,3142,4132,4312\})$ has the generating function
\[
\frac{1-6x+11x^2-5x^3}{(1-x)(1-3x)(1-3x+x^2)}.
\]
},
but we are interested instead in the simple permutations.  The additional rules for the words encoding simple permutations of length at least four are:

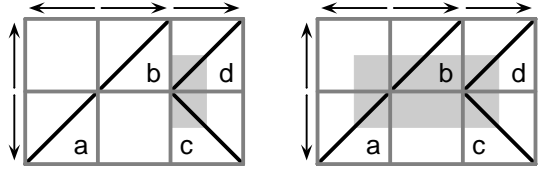
\begin{figure}
\begin{center}
\begin{tabular}{ccc}

\psset{xunit=0.0095in, yunit=0.0095in}
\psset{linewidth=0.005in}
\begin{pspicture}(-6,0)(120,88)
\psframe[linecolor=lightgray,fillstyle=solid,fillcolor=lightgray](80,20)(100,60)
\psline[linecolor=black,linestyle=solid,linewidth=0.02in]{c-c}(0,0)(80,80)
\psline[linecolor=black,linestyle=solid,linewidth=0.02in]{c-c}(80,40)(120,0)
\psline[linecolor=black,linestyle=solid,linewidth=0.02in]{c-c}(80,40)(120,80)
\psline[linecolor=darkgray,linestyle=solid,linewidth=0.02in]{c-c}(0,0)(0,80)
\psline[linecolor=darkgray,linestyle=solid,linewidth=0.02in]{c-c}(40,0)(40,80)
\psline[linecolor=darkgray,linestyle=solid,linewidth=0.02in]{c-c}(80,0)(80,80)
\psline[linecolor=darkgray,linestyle=solid,linewidth=0.02in]{c-c}(120,0)(120,80)
\psline[linecolor=darkgray,linestyle=solid,linewidth=0.02in]{c-c}(0,0)(120,0)
\psline[linecolor=darkgray,linestyle=solid,linewidth=0.02in]{c-c}(0,40)(120,40)
\psline[linecolor=darkgray,linestyle=solid,linewidth=0.02in]{c-c}(0,80)(120,80)
\psline[linecolor=black,linestyle=solid,linewidth=0.01in,arrowsize=0.06in]{<-c}(-6,2)(-6,38)
\psline[linecolor=black,linestyle=solid,linewidth=0.01in,arrowsize=0.06in]{c->}(-6,42)(-6,78)
\psline[linecolor=black,linestyle=solid,linewidth=0.01in,arrowsize=0.06in]{<-c}(2,86)(38,86)
\psline[linecolor=black,linestyle=solid,linewidth=0.01in,arrowsize=0.06in]{c->}(42,86)(78,86)
\psline[linecolor=black,linestyle=solid,linewidth=0.01in,arrowsize=0.06in]{c->}(82,86)(118,86)
\uput[135](40,0){$\ga$}
\uput[135](80,40){$\gb$}
\uput[45](80,0){$\gc$}
\uput[135](120,40){$\gd$}
\end{pspicture}

&&

\psset{xunit=0.0095in, yunit=0.0095in}
\psset{linewidth=0.005in}
\begin{pspicture}(-6,0)(120,88)
\psframe[linecolor=lightgray,fillstyle=solid,fillcolor=lightgray](20,20)(100,60)
\psline[linecolor=black,linestyle=solid,linewidth=0.02in]{c-c}(0,0)(80,80)
\psline[linecolor=black,linestyle=solid,linewidth=0.02in]{c-c}(80,40)(120,0)
\psline[linecolor=black,linestyle=solid,linewidth=0.02in]{c-c}(80,40)(120,80)
\psline[linecolor=darkgray,linestyle=solid,linewidth=0.02in]{c-c}(0,0)(0,80)
\psline[linecolor=darkgray,linestyle=solid,linewidth=0.02in]{c-c}(40,0)(40,80)
\psline[linecolor=darkgray,linestyle=solid,linewidth=0.02in]{c-c}(80,0)(80,80)
\psline[linecolor=darkgray,linestyle=solid,linewidth=0.02in]{c-c}(120,0)(120,80)
\psline[linecolor=darkgray,linestyle=solid,linewidth=0.02in]{c-c}(0,0)(120,0)
\psline[linecolor=darkgray,linestyle=solid,linewidth=0.02in]{c-c}(0,40)(120,40)
\psline[linecolor=darkgray,linestyle=solid,linewidth=0.02in]{c-c}(0,80)(120,80)
\psline[linecolor=black,linestyle=solid,linewidth=0.01in,arrowsize=0.06in]{<-c}(-6,2)(-6,38)
\psline[linecolor=black,linestyle=solid,linewidth=0.01in,arrowsize=0.06in]{c->}(-6,42)(-6,78)
\psline[linecolor=black,linestyle=solid,linewidth=0.01in,arrowsize=0.06in]{<-c}(2,86)(38,86)
\psline[linecolor=black,linestyle=solid,linewidth=0.01in,arrowsize=0.06in]{c->}(42,86)(78,86)
\psline[linecolor=black,linestyle=solid,linewidth=0.01in,arrowsize=0.06in]{c->}(82,86)(118,86)
\uput[135](40,0){$\ga$}
\uput[135](80,40){$\gb$}
\uput[45](80,0){$\gc$}
\uput[135](120,40){$\gd$}
\end{pspicture}

\end{tabular}
\end{center}
\caption{The shaded areas represent possible intervals in elements of the geometric grid class.}
\label{fig-4312-3142-grid-intervals}
\end{figure}

\begin{itemize}
\item To prevent intervals solely contained within an individual cell, we prohibit repetitions $\ga\ga$, $\gb\gb$, $\gc\gc$, or $\gd\gd$ as factors.
\item To prevent intervals of the form shown in the first pane of Figure~\ref{fig-4312-3142-grid-intervals}, we forbid words beginning with $\{\gc,\gd\}^2$.
\item To prevent intervals of the form shown in the second pane of Figure~\ref{fig-4312-3142-grid-intervals}, we forbid words of the form $\{\ga,\gb,\gc,\gd\}^\ast\{\ga,\gc,\gd\}^+$.
\end{itemize}
With these restrictions, we can then use the \textsf{automata} package~\cite{delgado:automata-----a-:} for \textsf{GAP}~\cite{:gap----groups-a:} to count the simple permutations of this grid class%
\footnote{These simple permutations have the generating function
\[
\frac{x+x^2-4x^3-3x^3}{(1+x)(1-2x)},
\]
showing that for $n\ge 3$ the simple permutations in this class are counting by the Jacobsthal numbers (\OEISlink{A001045} in the \OEISref).}.
For future reference, we record that the multivariate generating function for these words of length at least four which begin with $\ga$ is
\[
s(\xa,\xb,\xc,\xd)=\frac{\xa \xb \xc \xd}{1-\xa \xc-\xb \xd-\xc \xd-\xa \xc \xd- \xb \xc \xd},
\]
while the words of length at least four which begin with $\gc$ have multivariate generating function $\xc s(\xa,\xb,\xc,\xd)$.  Note that our rules preclude words encoding simple permutations from beginning with $\gb$ or $\gd$.

Now we characterise the inflations.  Because $3142$ is simple, it will not occur when inflating a $3142$-avoiding permutation by $3142$-avoiding intervals, so we need only avoid $4312$.  Since the class is sum closed, we have that $f_\oplus=f^2/(1+f)$, as in Section~\ref{sec-4213-3142}.  The skew decomposable permutations are a bit more complicated, but divide into a union:
\[
\left(\Av(21)\ominus\Av(312)\right)\cup\left(\Av_{\not\ominus}(4312,3142)\ominus\Av(12)\right),
\]
where $\Av_{\not\ominus}(4312,3142)$ denotes the set of skew indecomposable permutations in this class.  As the intersection of these two is simply $\Av(21)\ominus\Av(12)$, $f_\ominus=mc+(f-f_\ominus)m-m^2$, where
\[
m=\frac{x}{1-x}
\]
denotes the generating function for the nonempty decreasing (or, increasing) permutations.  Solving this shows
\[
f_\ominus
=
\frac{m(f+c-m)}{1+m}.
\]
Inflations of simple permutation of length at least four are a bit more complicated, as there are several cases.  In all such inflations, each entry which corresponds to a $\gb$ may only be inflated by an increasing permutation (but may be inflated by any such permutation), while each entry which corresponds to a $\gd$ may only be inflated by a $312$-avoiding permutation (but may be inflated by any such permutation).  If the word begins with a $\gc$, then it follows that the entry corresponding to the first $\gc$ may be inflated by any permutation in $\Av(312)$, while each subsequent entry corresponding to a $\gc$ may only be inflated by a decreasing permutation.  Otherwise, it follows from our rules that the word must begin with an $\ga$, and there are two cases.  If the entry corresponding to the first $\ga$ is inflated by a permutation containing a descent, then each entry corresponding to a $\gc$ must be inflated with a decreasing permutation.  Otherwise, if the entry corresponding to the first $\ga$ is inflated by an increasing permutation, then the entry corresponding to the first $\gc$ must be inflated by a permutation from $\Av(312)$, while each subsequent entry corresponding to a $\gc$ must be inflated by a decreasing permutation.  From our multivariate generating function for these simple permutations, it follows that the contribution of their inflations is
\[
\left(\frac{f-m}{f}+\frac{c}{f}+\frac{c}{m}\right)s(f,m,m,c)
=
\frac{cm^2(c-m+f+cf)}{1-2cm-cm^2-mf-cmf}.
\]
Combining this with the generating functions for $f_\oplus$ and $f_\ominus$ and solving for $f$ yields the generating function for the class (or, rather, its minimal polynomial).

\begin{theorem}\label{thm-4312-3142}
The generating function $f$ for $\Av(4312,3142)$ satisfies
\[
\begin{array}{rclcc}
(x^3-2x^2+x)f^4
&+&
(4x^3-9x^2+6x-1)f^3&&\\
&+&
(6x^3-12x^2+7x-1)f^2&&\\
&+&
(4x^3-5x^2+x)f&&\\
&+&
x^3
&=&0.
\end{array}
\]
\end{theorem}

The first several terms of this sequence are
\[
1,2,6,22,88,367,1568,6810,29943,132958,595227, 2683373,12170778,55499358,
\]
sequence \OEISlink{A165538} in the \OEISref. Though the form of the equation for $f$ is complicated, the close link with the Catalan numbers which can be seen in the previous development is enough to ensure that the radius of convergence is quite simple, exactly $1/5$, and in particular $f_n^{1/n} \to 5$.

\section{Example \#3: Avoiding $4231$ and $3124$}\label{sec-4231-3124}

Before our final example we review a well-studied class.  A permutation is {\it layered\/} if it is the direct sum of decreasing permutations (these decreasing permutations are called the {\it layers\/}).  The class of {\it layered permutations\/} has the basis $\{312, 231\}$.  To restrict the number of layers, we merely need to add an additional restriction, of the form $12\cdots k$.

\begin{proposition}\label{prop-4231-3124-simples}
The simple permutations of $\Av(4231, 3124)$ and $\Geom\fnmatrix{rrr}{0&1&-1\\1&-1&0}$ coincide.
\end{proposition}
\begin{proof}
First, it is straightforward to observe that
\[
\Geom\fnmatrix{rrr}{0&1&-1\\1&-1&0}
=
\Grid\fnmatrix{rrr}{0&1&-1\\1&-1&0}
\subseteq\Av(4231, 3124),
\]
so it suffices to prove that the simple permutations of $\Av(4231, 3124)$ are contained in this grid class.

Consider a simple permutation $\pi\in\Av(4231,3124)$ of length $n$.   We analyse the entries of $\pi$ to the left of $n$ and to the right of $n$ separately, beginning with the entries on the right.

\newtheorem*{claim-a}{\rm\bf Claim~\ref{prop-4231-3124-simples}.a}
\begin{claim-a}
The entries to the right of $n$ form a layered permutation with at most two layers.
\end{claim-a}

\newclaimproof{proof-claim-a}{Proof of Claim~\ref{prop-4231-3124-simples}.a}

\begin{proof-claim-a}
By the $4231$-avoidance of $\pi$ the entries to the right of $n$ avoid $231$, so to show that they are layered it suffices to show that they also avoid $312$.  Suppose otherwise.  Among all occurrences of $312$ choose one, $cab$, in which the `$1$' and `$2$' are as close together in position as possible (they will in fact be adjacent).  Since $\pi$ is simple, $\{a,b\}$ cannot be an interval and thus (because they are adjacent) must be separated vertically.  We claim that there is at least one such separator to the left of $n$.  Let $x$ denote an arbitrary separator of $\{a,b\}$.  We see that $x$ cannot lie horizontally between $n$ and $c$ by $4231$-avoidance. If $x$ were to lie horizontally between $c$ and $a$, then $\{x,a\}$ cannot be separated horizontally anywhere, nor vertically to the right of $n$, owing to the avoidance conditions, so must be separated vertically to the left of $n$.  This separator therefore separates $\{a,b\}$ vertically to the left of $n$, as desired.  The only other case is if $x$ lies to the right of $b$.  In this case choose $x$ to be the bottommost such separator.  Then $\{b,x\}$ must be separated.  This can only occur to the left of $n$ (giving the separator we desire) or to the right of $x$.  In this latter case it is easy to see that the region which consists of those points from $b$ to the right and which are also vertically between $x$ and $b$ (which contains $b$, $x$, and this new separator) can only be separated to the left of $n$, again giving the separator we desired.  Therefore in all cases we may assume that there is an entry, $x$, to the left of $n$ which vertically separates $a$ and $b$.  This situation is depicted on the left of Figure~\ref{fig-4231-3124-claim-a}.

Now consider $\{n,c\}$.  From Figure~\ref{fig-4231-3124-claim-a}, we see that these entries can only possibly be separated vertically by an entry to the left of $n$ and to the right of $x$.  Choose the leftmost such separator and label it $v$.  We now have the situation depicted in the centre of Figure \ref{fig-4231-3124-claim-a}.  However, it is now clear that the entries $\{c,n,v\}$ lie in a proper interval, contradicting the simplicity of $\pi$.  This contradiction shows that the entries to the right of $n$ must form a layered permutation.

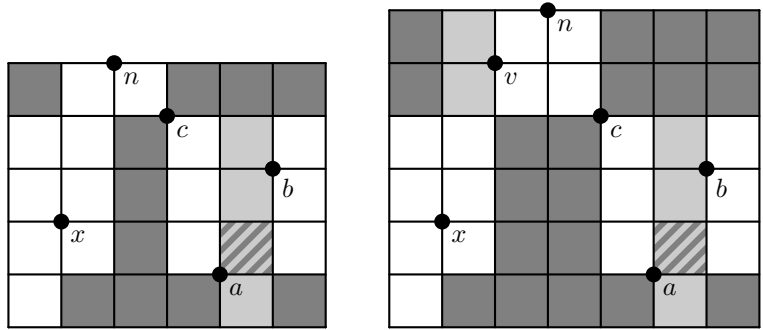
\begin{figure}[t]
\begin{center}
\begin{tabular}{ccc}
\psset{xunit=2pt, yunit=2pt, runit=1.5pt}
\begin{pspicture}(0,0)(60,50)
\pspolygon*[linecolor=darkgray](0,40)(10,40)(10,50)(0,50)
\pspolygon*[linecolor=darkgray](10,0)(20,0)(20,10)(10,10)
\pspolygon*[linecolor=darkgray](20,0)(30,0)(30,10)(20,10)
\pspolygon*[linecolor=darkgray](20,10)(30,10)(30,20)(20,20)
\pspolygon*[linecolor=darkgray](20,20)(30,20)(30,30)(20,30)
\pspolygon*[linecolor=darkgray](20,30)(30,30)(30,40)(20,40)
\pspolygon*[linecolor=darkgray](30,0)(40,0)(40,10)(30,10)
\pspolygon*[linecolor=darkgray](30,40)(40,40)(40,50)(30,50)
\pspolygon*[linecolor=lightgray](40,0)(50,0)(50,10)(40,10)
\pspolygon*[linecolor=lightgray](40,10)(50,10)(50,20)(40,20)
\psframe[linewidth=0,linecolor=white,fillstyle=hlines,hatchcolor=darkgray,hatchwidth=1.5,hatchsep=1.5](40,10)(50,20)
\pspolygon*[linecolor=lightgray](40,20)(50,20)(50,30)(40,30)
\pspolygon*[linecolor=lightgray](40,30)(50,30)(50,40)(40,40)
\pspolygon*[linecolor=darkgray](40,40)(50,40)(50,50)(40,50)
\pspolygon*[linecolor=darkgray](50,0)(60,0)(60,10)(50,10)
\pspolygon*[linecolor=darkgray](50,40)(60,40)(60,50)(50,50)
\pscircle*[linecolor=black](10,20){2.0}
\uput[-45](10,20){$x$}
\pscircle*[linecolor=black](20,50){2.0}
\uput[-45](20,50){$n$}
\pscircle*[linecolor=black](30,40){2.0}
\uput[-45](30,40){$c$}
\pscircle*[linecolor=black](40,10){2.0}
\uput[-45](40,10){$a$}
\pscircle*[linecolor=black](50,30){2.0}
\uput[-45](50,30){$b$}
\psline{c-c}(0,0)(0,50)
\psline{c-c}(0,0)(60,0)
\psline{c-c}(10,0)(10,50)
\psline{c-c}(0,10)(60,10)
\psline{c-c}(20,0)(20,50)
\psline{c-c}(0,20)(60,20)
\psline{c-c}(30,0)(30,50)
\psline{c-c}(0,30)(60,30)
\psline{c-c}(40,0)(40,50)
\psline{c-c}(0,40)(60,40)
\psline{c-c}(50,0)(50,50)
\psline{c-c}(0,50)(60,50)
\psline{c-c}(60,0)(60,50)
\end{pspicture}

&&

\psset{xunit=2pt, yunit=2pt, runit=1.5pt}
\begin{pspicture}(0,0)(70,60)
\pspolygon*[linecolor=darkgray](0,40)(10,40)(10,50)(0,50)
\pspolygon*[linecolor=darkgray](0,50)(10,50)(10,60)(0,60)
\pspolygon*[linecolor=darkgray](10,0)(20,0)(20,10)(10,10)
\pspolygon*[linecolor=lightgray](10,40)(20,40)(20,50)(10,50)
\pspolygon*[linecolor=lightgray](10,50)(20,50)(20,60)(10,60)
\pspolygon*[linecolor=darkgray](20,0)(30,0)(30,10)(20,10)
\pspolygon*[linecolor=darkgray](20,10)(30,10)(30,20)(20,20)
\pspolygon*[linecolor=darkgray](20,20)(30,20)(30,30)(20,30)
\pspolygon*[linecolor=darkgray](20,30)(30,30)(30,40)(20,40)
\pspolygon*[linecolor=darkgray](30,0)(40,0)(40,10)(30,10)
\pspolygon*[linecolor=darkgray](30,10)(40,10)(40,20)(30,20)
\pspolygon*[linecolor=darkgray](30,20)(40,20)(40,30)(30,30)
\pspolygon*[linecolor=darkgray](30,30)(40,30)(40,40)(30,40)
\pspolygon*[linecolor=darkgray](40,0)(50,0)(50,10)(40,10)
\pspolygon*[linecolor=darkgray](40,40)(50,40)(50,50)(40,50)
\pspolygon*[linecolor=darkgray](40,50)(50,50)(50,60)(40,60)
\pspolygon*[linecolor=lightgray](50,0)(60,0)(60,10)(50,10)
\pspolygon*[linecolor=lightgray](50,10)(60,10)(60,20)(50,20)
\psframe[linewidth=0,linecolor=white,fillstyle=hlines,hatchcolor=darkgray,hatchwidth=1.5,hatchsep=1.5](50,10)(60,20)
\pspolygon*[linecolor=lightgray](50,20)(60,20)(60,30)(50,30)
\pspolygon*[linecolor=lightgray](50,30)(60,30)(60,40)(50,40)
\pspolygon*[linecolor=darkgray](50,40)(60,40)(60,50)(50,50)
\pspolygon*[linecolor=darkgray](50,50)(60,50)(60,60)(50,60)
\pspolygon*[linecolor=darkgray](60,0)(70,0)(70,10)(60,10)
\pspolygon*[linecolor=darkgray](60,40)(70,40)(70,50)(60,50)
\pspolygon*[linecolor=darkgray](60,50)(70,50)(70,60)(60,60)
\pscircle*[linecolor=black](10,20){2.0}
\uput[-45](10,20){$x$}
\pscircle*[linecolor=black](20,50){2.0}
\uput[-45](20,50){$v$}
\pscircle*[linecolor=black](30,60){2.0}
\uput[-45](30,60){$n$}
\pscircle*[linecolor=black](40,40){2.0}
\uput[-45](40,40){$c$}
\pscircle*[linecolor=black](50,10){2.0}
\uput[-45](50,10){$a$}
\pscircle*[linecolor=black](60,30){2.0}
\uput[-45](60,30){$b$}
\psline{c-c}(0,0)(0,60)
\psline{c-c}(0,0)(70,0)
\psline{c-c}(10,0)(10,60)
\psline{c-c}(0,10)(70,10)
\psline{c-c}(20,0)(20,60)
\psline{c-c}(0,20)(70,20)
\psline{c-c}(30,0)(30,60)
\psline{c-c}(0,30)(70,30)
\psline{c-c}(40,0)(40,60)
\psline{c-c}(0,40)(70,40)
\psline{c-c}(50,0)(50,60)
\psline{c-c}(0,50)(70,50)
\psline{c-c}(60,0)(60,60)
\psline{c-c}(0,60)(70,60)
\psline{c-c}(70,0)(70,60)
\end{pspicture}
\end{tabular}
\end{center}
\caption{The end games in the proof of Claim \ref{prop-4231-3124-simples}.a.}
\label{fig-4231-3124-claim-a}
\end{figure}

Having established that these entries are layered, it is easy to see that there are at most two layers.  Otherwise the entries to the right of $n$ would contain a copy of $123$.  The `$2$' and `$3$' in this copy of $123$ must be separated by an entry to the left of $n$ (because the entries to the right of $n$ form a layered permutation), but this would create a copy of $3124$.  This contradiction completes the proof of Claim~\ref{prop-4231-3124-simples}.a.
\end{proof-claim-a}

\newtheorem*{claim-b}{\rm\bf Claim~\ref{prop-4231-3124-simples}.b}
\begin{claim-b}
The entries to the left of $n$ and above $\pi(n)$ are increasing.
\end{claim-b}

\newclaimproof{proof-claim-b}{Proof of Claim~\ref{prop-4231-3124-simples}.b}

\begin{proof-claim-b}
Let $a=\pi(n)$ and suppose to the contrary that the entries to the left of $n$ and above $a$ contain an inversion.  Choose such an inversion $yx$ with $y$ as far left as possible and $x$ as close to $y$ as possible.  This gives the situation depicted on the left of Figure~\ref{fig-4231-3124-claim-b}.  As can be seen in this diagram, $\{x,y\}$ could only possibly be separated horizontally.  Let $z$ denote a topmost such separator.  This gives the situation depicted on the right of Figure~\ref{fig-4231-3124-claim-b}.  However, as this diagram indicates, $y$ and $z$ now belong to a proper interval, contradicting the simplicity of $\pi$.
\end{proof-claim-b}

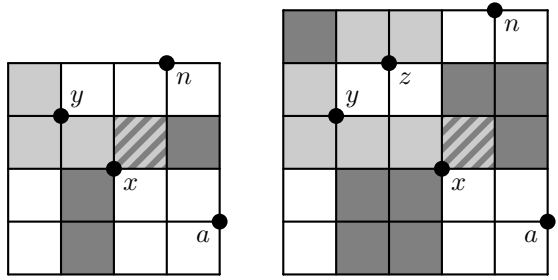
\begin{figure}[t]
\begin{center}
\begin{tabular}{ccc}
\psset{xunit=2pt, yunit=2pt, runit=1.5pt}
\begin{pspicture}(0,0)(40,40)
\pspolygon*[linecolor=lightgray](0,20)(10,20)(10,30)(0,30)
\pspolygon*[linecolor=lightgray](0,30)(10,30)(10,40)(0,40)
\pspolygon*[linecolor=darkgray](10,0)(20,0)(20,10)(10,10)
\pspolygon*[linecolor=darkgray](10,10)(20,10)(20,20)(10,20)
\pspolygon*[linecolor=lightgray](10,20)(20,20)(20,30)(10,30)
\pspolygon*[linecolor=lightgray](20,20)(30,20)(30,30)(20,30)
\psframe[linewidth=0,linecolor=white,fillstyle=hlines,hatchcolor=darkgray,hatchwidth=1.5,hatchsep=1.5](20,20)(30,30)
\pspolygon*[linecolor=darkgray](30,20)(40,20)(40,30)(30,30)
\pscircle*[linecolor=black](10,30){2.0}
\uput[45](10,30){$y$}
\pscircle*[linecolor=black](20,20){2.0}
\uput[-45](20,20){$x$}
\pscircle*[linecolor=black](30,40){2.0}
\uput[-45](30,40){$n$}
\pscircle*[linecolor=black](40,10){2.0}
\uput[225](40,10){$a$}
\psline{c-c}(0,0)(0,40)
\psline{c-c}(0,0)(40,0)
\psline{c-c}(10,0)(10,40)
\psline{c-c}(0,10)(40,10)
\psline{c-c}(20,0)(20,40)
\psline{c-c}(0,20)(40,20)
\psline{c-c}(30,0)(30,40)
\psline{c-c}(0,30)(40,30)
\psline{c-c}(40,0)(40,40)
\psline{c-c}(0,40)(40,40)
\end{pspicture}

&&

\psset{xunit=2pt, yunit=2pt, runit=1.5pt}
\begin{pspicture}(0,0)(50,50)
\pspolygon*[linecolor=lightgray](0,20)(10,20)(10,30)(0,30)
\pspolygon*[linecolor=lightgray](0,30)(10,30)(10,40)(0,40)
\pspolygon*[linecolor=darkgray](0,40)(10,40)(10,50)(0,50)
\pspolygon*[linecolor=darkgray](10,0)(20,0)(20,10)(10,10)
\pspolygon*[linecolor=darkgray](10,10)(20,10)(20,20)(10,20)
\pspolygon*[linecolor=lightgray](10,20)(20,20)(20,30)(10,30)
\pspolygon*[linecolor=lightgray](10,40)(20,40)(20,50)(10,50)
\pspolygon*[linecolor=darkgray](20,0)(30,0)(30,10)(20,10)
\pspolygon*[linecolor=darkgray](20,10)(30,10)(30,20)(20,20)
\pspolygon*[linecolor=lightgray](20,20)(30,20)(30,30)(20,30)
\pspolygon*[linecolor=lightgray](20,40)(30,40)(30,50)(20,50)
\pspolygon*[linecolor=lightgray](30,20)(40,20)(40,30)(30,30)
\psframe[linewidth=0,linecolor=white,fillstyle=hlines,hatchcolor=darkgray,hatchwidth=1.5,hatchsep=1.5](30,20)(40,30)
\pspolygon*[linecolor=darkgray](30,30)(40,30)(40,40)(30,40)
\pspolygon*[linecolor=darkgray](40,20)(50,20)(50,30)(40,30)
\pspolygon*[linecolor=darkgray](40,30)(50,30)(50,40)(40,40)
\pscircle*[linecolor=black](10,30){2.0}
\uput[45](10,30){$y$}
\pscircle*[linecolor=black](20,40){2.0}
\uput[-45](20,40){$z$}
\pscircle*[linecolor=black](30,20){2.0}
\uput[-45](30,20){$x$}
\pscircle*[linecolor=black](40,50){2.0}
\uput[-45](40,50){$n$}
\pscircle*[linecolor=black](50,10){2.0}
\uput[225](50,10){$a$}
\psline{c-c}(0,0)(0,50)
\psline{c-c}(0,0)(50,0)
\psline{c-c}(10,0)(10,50)
\psline{c-c}(0,10)(50,10)
\psline{c-c}(20,0)(20,50)
\psline{c-c}(0,20)(50,20)
\psline{c-c}(30,0)(30,50)
\psline{c-c}(0,30)(50,30)
\psline{c-c}(40,0)(40,50)
\psline{c-c}(0,40)(50,40)
\psline{c-c}(50,0)(50,50)
\psline{c-c}(0,50)(50,50)
\end{pspicture}
\end{tabular}
\end{center}
\caption{The end games in the proof of Claim \ref{prop-4231-3124-simples}.b.}
\label{fig-4231-3124-claim-b}
\end{figure}

\newtheorem*{claim-c}{\rm\bf Claim~\ref{prop-4231-3124-simples}.c}
\begin{claim-c}
Let $c$ denote the leftmost entry of $\pi$ greater than $\pi(n)$ (note that $c$ may equal $n$). The entries to the left of $c$ (which lie below $\pi(n)$ by Claim~\ref{prop-4231-3124-simples}.b) are increasing.
\end{claim-c}

\newclaimproof{proof-claim-c}{Proof of Claim~\ref{prop-4231-3124-simples}.c}

\begin{proof-claim-c}
The proof follows well-travelled lines. Suppose to the contrary that there is an inversion among these entries, and choose one such $ba$ where $b$ is as far left as possible, and $a$ is as small as possible. The cell bounded by $\{a,b\}$ can only be split above, and we may choose a split point $d$ which is as large as possible. Now the box bounded by $\{a,b,d\}$ defines a proper interval, a contradiction.
\end{proof-claim-c}

We are now in position to complete the proof of the proposition, after a brief recap of the structure we have established.  Let $b=\pi(n)$, let $c$ denote the leftmost entry which is greater than $b$ (note that $c=n$ is a possibility), and finally let $a$ denote the bottommost entry which lies horizontally between $c$ and $n$ (such an entry need not exist, but this does not affect the argument).  We then have the situation depicted on the left of Figure~\ref{fig-4231-3124-finale}.

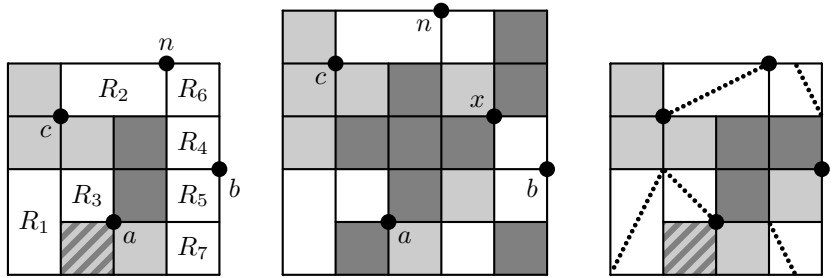
\begin{figure}[t]
\begin{center}
\begin{tabular}{ccccccc}

\psset{xunit=2pt, yunit=2pt, runit=1.5pt}
\begin{pspicture}(0,0)(40,40)
\pspolygon*[linecolor=lightgray](0,20)(10,20)(10,30)(0,30)
\pspolygon*[linecolor=lightgray](0,30)(10,30)(10,40)(0,40)
\pspolygon*[linecolor=lightgray](10,0)(20,0)(20,10)(10,10)
\psframe[linewidth=0,linecolor=white,fillstyle=hlines,hatchcolor=darkgray,hatchwidth=1.5,hatchsep=1.5](10,0)(20,10)
\pspolygon*[linecolor=lightgray](10,20)(20,20)(20,30)(10,30)
\pspolygon*[linecolor=lightgray](20,0)(30,0)(30,10)(20,10)
\pspolygon*[linecolor=darkgray](20,10)(30,10)(30,20)(20,20)
\pspolygon*[linecolor=darkgray](20,20)(30,20)(30,30)(20,30)
\pscircle*[linecolor=black](10,30){2.0}
\uput[-135](10,30){$c$}
\pscircle*[linecolor=black](20,10){2.0}
\uput[-45](20,10){$a$}
\pscircle*[linecolor=black](30,40){2.0}
\uput[90](30,40){$n$}
\pscircle*[linecolor=black](40,20){2.0}
\uput[-45](40,20){$b$}
\rput[c](5,10){$R_1$}
\rput[c](20,35){$R_2$}
\rput[c](15,15){$R_3$}
\rput[c](35,25){$R_4$}
\rput[c](35,15){$R_5$}
\rput[c](35,35){$R_6$}
\rput[c](35,5){$R_7$}
\psline{c-c}(0,0)(0,40)
\psline{c-c}(0,0)(40,0)
\psline{c-c}(10,0)(10,40)
\psline{c-c}(10,10)(40,10)
\psline{c-c}(20,0)(20,30)
\psline{c-c}(0,20)(40,20)
\psline{c-c}(30,0)(30,40)
\psline{c-c}(0,30)(40,30)
\psline{c-c}(40,0)(40,40)
\psline{c-c}(0,40)(40,40)
\end{pspicture}

&&

\psset{xunit=2pt, yunit=2pt, runit=1.5pt}
\begin{pspicture}(0,0)(50,50)
\pspolygon*[linecolor=lightgray](0,20)(10,20)(10,30)(0,30)
\pspolygon*[linecolor=lightgray](0,30)(10,30)(10,40)(0,40)
\pspolygon*[linecolor=lightgray](0,40)(10,40)(10,50)(0,50)
\pspolygon*[linecolor=darkgray](10,0)(20,0)(20,10)(10,10)
\pspolygon*[linecolor=darkgray](10,20)(20,20)(20,30)(10,30)
\pspolygon*[linecolor=lightgray](10,30)(20,30)(20,40)(10,40)
\pspolygon*[linecolor=lightgray](20,0)(30,0)(30,10)(20,10)
\pspolygon*[linecolor=darkgray](20,10)(30,10)(30,20)(20,20)
\pspolygon*[linecolor=darkgray](20,20)(30,20)(30,30)(20,30)
\pspolygon*[linecolor=darkgray](20,30)(30,30)(30,40)(20,40)
\pspolygon*[linecolor=lightgray](30,10)(40,10)(40,20)(30,20)
\pspolygon*[linecolor=darkgray](30,20)(40,20)(40,30)(30,30)
\pspolygon*[linecolor=lightgray](30,30)(40,30)(40,40)(30,40)
\pspolygon*[linecolor=darkgray](40,0)(50,0)(50,10)(40,10)
\pspolygon*[linecolor=darkgray](40,30)(50,30)(50,40)(40,40)
\pspolygon*[linecolor=darkgray](40,40)(50,40)(50,50)(40,50)
\pscircle*[linecolor=black](10,40){2.0}
\uput[225](10,40){$c$}
\pscircle*[linecolor=black](20,10){2.0}
\uput[-45](20,10){$a$}
\pscircle*[linecolor=black](30,50){2.0}
\uput[225](30,50){$n$}
\pscircle*[linecolor=black](40,30){2.0}
\uput[135](40,30){$x$}
\pscircle*[linecolor=black](50,20){2.0}
\uput[-135](50,20){$b$}
\psline{c-c}(0,0)(0,50)
\psline{c-c}(0,0)(50,0)
\psline{c-c}(10,0)(10,50)
\psline{c-c}(10,10)(50,10)
\psline{c-c}(20,0)(20,40)
\psline{c-c}(0,20)(50,20)
\psline{c-c}(30,0)(30,50)
\psline{c-c}(0,30)(50,30)
\psline{c-c}(40,0)(40,50)
\psline{c-c}(0,40)(50,40)
\psline{c-c}(50,0)(50,50)
\psline{c-c}(0,50)(50,50)
\end{pspicture}

&&

\psset{xunit=2pt, yunit=2pt, runit=1.5pt}
\begin{pspicture}(0,0)(40,40)
\pspolygon*[linecolor=lightgray](0,20)(10,20)(10,30)(0,30)
\pspolygon*[linecolor=lightgray](0,30)(10,30)(10,40)(0,40)
\pspolygon*[linecolor=lightgray](10,0)(20,0)(20,10)(10,10)
\psframe[linewidth=0,linecolor=white,fillstyle=hlines,hatchcolor=darkgray,hatchwidth=1.5,hatchsep=1.5](10,0)(20,10)
\pspolygon*[linecolor=lightgray](10,20)(20,20)(20,30)(10,30)
\pspolygon*[linecolor=lightgray](20,0)(30,0)(30,10)(20,10)
\pspolygon*[linecolor=darkgray](20,10)(30,10)(30,20)(20,20)
\pspolygon*[linecolor=darkgray](20,20)(30,20)(30,30)(20,30)
\pspolygon*[linecolor=lightgray](30,10)(40,10)(40,20)(30,20)
\pspolygon*[linecolor=darkgray](30,20)(40,20)(40,30)(30,30)
\pscircle*[linecolor=black](10,30){2.0}
\pscircle*[linecolor=black](20,10){2.0}
\pscircle*[linecolor=black](30,40){2.0}
\pscircle*[linecolor=black](40,20){2.0}
\psline{c-c}(0,0)(0,40)
\psline{c-c}(0,0)(40,0)
\psline{c-c}(10,0)(10,40)
\psline{c-c}(10,10)(40,10)
\psline{c-c}(20,0)(20,30)
\psline{c-c}(0,20)(40,20)
\psline{c-c}(30,0)(30,40)
\psline{c-c}(0,30)(40,30)
\psline{c-c}(40,0)(40,40)
\psline{c-c}(0,40)(40,40)
\psline[linestyle=dotted,linewidth=1.2,dotsep=0.5](0.5,1)(10,20)
\psline[linestyle=dotted,linewidth=1.2,dotsep=0.5](10,20)(20,10)
\psline[linestyle=dotted,linewidth=1.2,dotsep=0.5](10,30)(30,40)
\psline[linestyle=dotted,linewidth=1.2,dotsep=0.5](30,10)(35,0)
\psline[linestyle=dotted,linewidth=1.2,dotsep=0.5](35,40)(40,30)
\end{pspicture}

\end{tabular}
\end{center}

\caption{Final considerations in the proof of Proposition~\ref{prop-4231-3124-simples}.}
\label{fig-4231-3124-finale}
\end{figure}

The first three labeled regions are further restricted as follows:
\begin{itemize}
\item $R_1$ must be increasing by Claim~\ref{prop-4231-3124-simples}.c.
\item $R_2$ must be increasing by Claim~\ref{prop-4231-3124-simples}.b.
\item $R_3$ must be decreasing because $\pi$ avoids $3124$.
\end{itemize}
Next we claim that $R_4$ and $R_5$ are both empty.  First suppose to the contrary that $R_4$ is nonempty, and take $x$ to be the topmost entry in this region.  We then have the situation depicted in the centre of Figure~\ref{fig-4231-3124-finale}, which shows that $\{b,x\}$ can only be separated by an entry in $R_5$.  Let $y$ denote the bottommost such separator.  It can then be seen that there is no way to separate $\{b,x,y\}$.  Showing that $R_5$ is empty is very similar.  Suppose to the contrary that this region is nonempty and let $x$ denote the bottommost entry in the region.  It can be seen that there are two ways to separate $\{b,x\}$: vertically with an entry in $R_5$ or horizontally with an entry in $R_4$.  In each case, though, these new entries cannot be separated from either $b$ or $x$.

Finally, regions $R_6$ and $R_7$ must both be decreasing because $\pi$ avoids $4231$ and $3124$ respectively.  Then Claim~\ref{prop-4231-3124-simples}.a shows that regions $R_5$ and $R_6$ must together form a layered permutation with at most two layers.  The structure of $\pi$ is now displayed on the right of Figure~\ref{fig-4231-3124-finale}, which shows that $\pi$ does indeed lie in the grid class desired, completing the proof.
\end{proof}

Having restricted our attention to this (geometric) grid class, we now seek to place it in bijection with a regular language, following points (G1) and (G2) of Section~\ref{sec-4312-3142}.  The first is the easiest to deal with since we need only forbid factors of the forms $\{\gc,\gd\}^+\ga$ and $\gd^+\gb$.

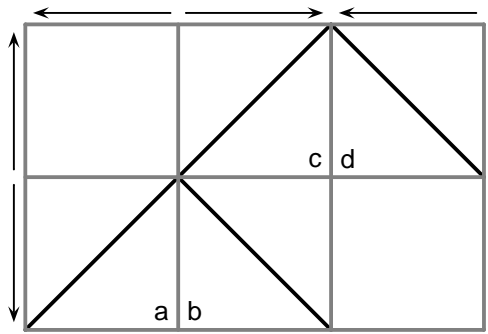
\begin{figure}
\begin{center}
\psset{xunit=0.02in, yunit=0.02in}
\psset{linewidth=0.005in}
\begin{pspicture}(-3,0)(120,85)
\psline[linecolor=black,linestyle=solid,linewidth=0.02in]{c-c}(0,0)(40,40)
\psline[linecolor=black,linestyle=solid,linewidth=0.02in]{c-c}(40,40)(80,80)
\psline[linecolor=black,linestyle=solid,linewidth=0.02in]{c-c}(40,40)(80,0)
\psline[linecolor=black,linestyle=solid,linewidth=0.02in]{c-c}(80,80)(120,40)
\psline[linecolor=darkgray,linestyle=solid,linewidth=0.02in]{c-c}(0,0)(0,80)
\psline[linecolor=darkgray,linestyle=solid,linewidth=0.02in]{c-c}(40,0)(40,80)
\psline[linecolor=darkgray,linestyle=solid,linewidth=0.02in]{c-c}(80,0)(80,80)
\psline[linecolor=darkgray,linestyle=solid,linewidth=0.02in]{c-c}(120,0)(120,80)
\psline[linecolor=darkgray,linestyle=solid,linewidth=0.02in]{c-c}(0,0)(120,0)
\psline[linecolor=darkgray,linestyle=solid,linewidth=0.02in]{c-c}(0,40)(120,40)
\psline[linecolor=darkgray,linestyle=solid,linewidth=0.02in]{c-c}(0,80)(120,80)
\psline[linecolor=black,linestyle=solid,linewidth=0.01in,arrowsize=0.06in]{<-c}(-3,2)(-3,38)
\psline[linecolor=black,linestyle=solid,linewidth=0.01in,arrowsize=0.06in]{c->}(-3,42)(-3,78)
\psline[linecolor=black,linestyle=solid,linewidth=0.01in,arrowsize=0.06in]{<-c}(2,83)(38,83)
\psline[linecolor=black,linestyle=solid,linewidth=0.01in,arrowsize=0.06in]{c->}(42,83)(78,83)
\psline[linecolor=black,linestyle=solid,linewidth=0.01in,arrowsize=0.06in]{<-c}(82,83)(118,83)
\uput[135](40,0){$\ga$}
\uput[45](40,0){$\gb$}
\uput[135](80,40){$\gc$}
\uput[45](80,40){$\gd$}
\end{pspicture}
\end{center}
\caption{A choice of cell alphabet and row and column signs for the geometric grid class of interest.}
\label{fig-4231-3124-grid-encoding}
\end{figure}

To handle (G2), we use the same preference for $M$-griddings as in Section~\ref{sec-4312-3142}: among all $M$-griddings of a permutation, we prefer the one that has the most entries in the first column, then the most entries in the second column, and then the most entries in the first row.  The words that correspond to these preferred griddings are those which do not begin with $\gb$ or $\ga^\ast\gc$, do not end with $\gd$, and are not of the forms $\gd\{\ga,\gb\}^\ast$ or $\ga^\ast\{\gc,\gd\}^+$.

This language allows us to enumerate the grid class itself%
\footnote{The generating function for this grid class is
\[
\frac{1-5x+7x^2-x^3}{(1-x)(1-2x)(1-3x)},
\]
(sequence \OEISlink{A083323} in the \OEISref).  Our computations suggest that the basis of this class is
\[
\{4312, 4231, 4123, 3124, 32541, 21534, 21435\},
\]
but we have not verified this with a formal proof.},
and now we restrict to encoding the simple permutations.

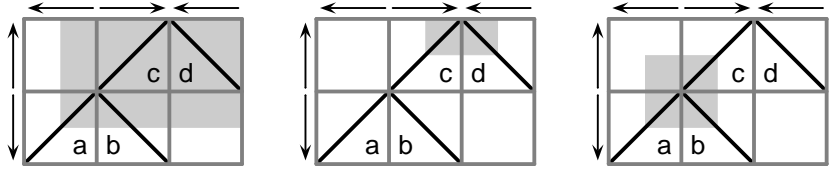
\begin{figure}
\begin{center}
\begin{tabular}{ccccc}

\psset{xunit=0.0095in, yunit=0.0095in}
\psset{linewidth=0.005in}
\begin{pspicture}(-6,0)(120,88)
\psframe[linecolor=lightgray,fillstyle=solid,fillcolor=lightgray](20,20)(120,80)
\psline[linecolor=black,linestyle=solid,linewidth=0.02in](0,0)(40,40)
\psline[linecolor=black,linestyle=solid,linewidth=0.02in](40,40)(80,80)
\psline[linecolor=black,linestyle=solid,linewidth=0.02in](40,40)(80,0)
\psline[linecolor=black,linestyle=solid,linewidth=0.02in](80,80)(120,40)
\psline[linecolor=darkgray,linestyle=solid,linewidth=0.02in]{c-c}(0,0)(0,80)
\psline[linecolor=darkgray,linestyle=solid,linewidth=0.02in]{c-c}(40,0)(40,80)
\psline[linecolor=darkgray,linestyle=solid,linewidth=0.02in]{c-c}(80,0)(80,80)
\psline[linecolor=darkgray,linestyle=solid,linewidth=0.02in]{c-c}(120,0)(120,80)
\psline[linecolor=darkgray,linestyle=solid,linewidth=0.02in]{c-c}(0,0)(120,0)
\psline[linecolor=darkgray,linestyle=solid,linewidth=0.02in]{c-c}(0,40)(120,40)
\psline[linecolor=darkgray,linestyle=solid,linewidth=0.02in]{c-c}(0,80)(120,80)
\psline[linecolor=black,linestyle=solid,linewidth=0.01in,arrowsize=0.06in]{<-c}(-6,2)(-6,38)
\psline[linecolor=black,linestyle=solid,linewidth=0.01in,arrowsize=0.06in]{c->}(-6,42)(-6,78)
\psline[linecolor=black,linestyle=solid,linewidth=0.01in,arrowsize=0.06in]{<-c}(2,86)(38,86)
\psline[linecolor=black,linestyle=solid,linewidth=0.01in,arrowsize=0.06in]{c->}(42,86)(78,86)
\psline[linecolor=black,linestyle=solid,linewidth=0.01in,arrowsize=0.06in]{<-c}(82,86)(118,86)
\uput[135](40,0){$\ga$}
\uput[45](40,0){$\gb$}
\uput[135](80,40){$\gc$}
\uput[45](80,40){$\gd$}
\end{pspicture}

&&

\psset{xunit=0.0095in, yunit=0.0095in}
\psset{linewidth=0.005in}
\begin{pspicture}(-6,0)(120,88)
\psframe[linecolor=lightgray,fillstyle=solid,fillcolor=lightgray](60,60)(100,80)
\psline[linecolor=black,linestyle=solid,linewidth=0.02in](0,0)(40,40)
\psline[linecolor=black,linestyle=solid,linewidth=0.02in](40,40)(80,80)
\psline[linecolor=black,linestyle=solid,linewidth=0.02in](40,40)(80,0)
\psline[linecolor=black,linestyle=solid,linewidth=0.02in](80,80)(120,40)
\psline[linecolor=darkgray,linestyle=solid,linewidth=0.02in]{c-c}(0,0)(0,80)
\psline[linecolor=darkgray,linestyle=solid,linewidth=0.02in]{c-c}(40,0)(40,80)
\psline[linecolor=darkgray,linestyle=solid,linewidth=0.02in]{c-c}(80,0)(80,80)
\psline[linecolor=darkgray,linestyle=solid,linewidth=0.02in]{c-c}(120,0)(120,80)
\psline[linecolor=darkgray,linestyle=solid,linewidth=0.02in]{c-c}(0,0)(120,0)
\psline[linecolor=darkgray,linestyle=solid,linewidth=0.02in]{c-c}(0,40)(120,40)
\psline[linecolor=darkgray,linestyle=solid,linewidth=0.02in]{c-c}(0,80)(120,80)
\psline[linecolor=black,linestyle=solid,linewidth=0.01in,arrowsize=0.06in]{<-c}(-6,2)(-6,38)
\psline[linecolor=black,linestyle=solid,linewidth=0.01in,arrowsize=0.06in]{c->}(-6,42)(-6,78)
\psline[linecolor=black,linestyle=solid,linewidth=0.01in,arrowsize=0.06in]{<-c}(2,86)(38,86)
\psline[linecolor=black,linestyle=solid,linewidth=0.01in,arrowsize=0.06in]{c->}(42,86)(78,86)
\psline[linecolor=black,linestyle=solid,linewidth=0.01in,arrowsize=0.06in]{<-c}(82,86)(118,86)
\uput[135](40,0){$\ga$}
\uput[45](40,0){$\gb$}
\uput[135](80,40){$\gc$}
\uput[45](80,40){$\gd$}
\end{pspicture}

&&

\psset{xunit=0.0095in, yunit=0.0095in}
\psset{linewidth=0.005in}
\begin{pspicture}(-6,0)(120,88)
\psframe[linecolor=lightgray,fillstyle=solid,fillcolor=lightgray](20,20)(60,60)
\psline[linecolor=black,linestyle=solid,linewidth=0.02in](0,0)(40,40)
\psline[linecolor=black,linestyle=solid,linewidth=0.02in](40,40)(80,80)
\psline[linecolor=black,linestyle=solid,linewidth=0.02in](40,40)(80,0)
\psline[linecolor=black,linestyle=solid,linewidth=0.02in](80,80)(120,40)
\psline[linecolor=darkgray,linestyle=solid,linewidth=0.02in]{c-c}(0,0)(0,80)
\psline[linecolor=darkgray,linestyle=solid,linewidth=0.02in]{c-c}(40,0)(40,80)
\psline[linecolor=darkgray,linestyle=solid,linewidth=0.02in]{c-c}(80,0)(80,80)
\psline[linecolor=darkgray,linestyle=solid,linewidth=0.02in]{c-c}(120,0)(120,80)
\psline[linecolor=darkgray,linestyle=solid,linewidth=0.02in]{c-c}(0,0)(120,0)
\psline[linecolor=darkgray,linestyle=solid,linewidth=0.02in]{c-c}(0,40)(120,40)
\psline[linecolor=darkgray,linestyle=solid,linewidth=0.02in]{c-c}(0,80)(120,80)
\psline[linecolor=black,linestyle=solid,linewidth=0.01in,arrowsize=0.06in]{<-c}(-6,2)(-6,38)
\psline[linecolor=black,linestyle=solid,linewidth=0.01in,arrowsize=0.06in]{c->}(-6,42)(-6,78)
\psline[linecolor=black,linestyle=solid,linewidth=0.01in,arrowsize=0.06in]{<-c}(2,86)(38,86)
\psline[linecolor=black,linestyle=solid,linewidth=0.01in,arrowsize=0.06in]{c->}(42,86)(78,86)
\psline[linecolor=black,linestyle=solid,linewidth=0.01in,arrowsize=0.06in]{<-c}(82,86)(118,86)
\uput[135](40,0){$\ga$}
\uput[45](40,0){$\gb$}
\uput[135](80,40){$\gc$}
\uput[45](80,40){$\gd$}
\end{pspicture}

\end{tabular}
\end{center}
\caption{The shaded areas represent possible intervals in elements of the geometric grid class.}
\label{fig-4231-3124-grid-intervals}
\end{figure}

\begin{itemize}
\item To prevent intervals contained within individual cells, we prohibit repetitions $\ga\ga$, $\gb\gb$, $\gc\gc$, $\gd\gd$ as factors.
\item To prevent intervals of the form shown in the first pane of Figure~\ref{fig-4231-3124-grid-intervals}, we insist that the last $\ga$ is followed by a $\gb$, i.e., we prohibit words which end in $\ga\{\gc,\gd\}^\ast$,
\item To prevent intervals of the form shown in the second pane of Figure~\ref{fig-4231-3124-grid-intervals}, we require that there is a $\gb$ after the last letter in $\{\gc,\gd\}$; by the previous two rules, this means we need only prohibit words which end in $\gc\gd$ or $\gd\gc$.
\item To prevent intervals of the form shown in the third pane of Figure~\ref{fig-4231-3124-grid-intervals}, we prohibit words which begin with two or more letters from $\{\ga,\gb,\gc\}$.
\end{itemize}
Finally, we exclude the word $\gd\gc\gb$, which is mapped by $\bij$ to the nonsimple permutation $312$.

We can then compute (again using the \textsf{automata} package~\cite{delgado:automata-----a-:} for \textsf{GAP}~\cite{:gap----groups-a:}) that the multivariate generating function (counting occurrences of each letter) for the simple permutations in this geometric grid class is
\[
s(\xa,\xb,\xc,\xd)=\frac{\xb\xc\xd(\xa+\xc+\xa\xb+\xa\xc+\xb\xc+\xc\xd+\xa\xb\xc+\xb\xc\xd)}
{1-\xa\xb-\xb\xc-\xc\xd-\xa\xb\xc-\xb\xc\xd}.
\]
It remains only to determine the allowed inflations.  First, the sum decomposable permutations all have unique representations in the form $\Av_{\not\oplus}(312)\oplus \Av(4231,3124)$, where $\Av_{\not\oplus}(312)$ denotes the sum indecomposable permutations in $\Av(312)$.  As $\Av_{\not\oplus}(312)$ is well-known to be enumerated by the shifted Catalan numbers, we get
\[
f_\oplus=(xc+x)f.
\]

The skew decomposable permutations are of the form $\Av(312)\ominus \Av(231,3124)$, and so (deviating slightly from our usual conventions) they all have a unique representation of the form $\Av(312)\ominus \Av_{\not\ominus}(231,3124)$.  The skew indecomposable permutations in $\Av(231,3124)$ are counted by the Fibonacci numbers of odd index ($1$, $1$, $3$, $8$, $21$, $\dots$), giving that
\[
f_\ominus=\frac{x-2x^2+x^2}{1-3x+x^2}c.
\]
We now come to inflations of simple permutations of length at least four.  Each entry corresponding to an $\ga$ or $\gc$ may be inflated by an arbitrary member of $\Av(312)$.  Each entry corresponding to a $\gb$ or non-initial $\gd$ may be inflated only by a decreasing interval.  Finally, the first entry corresponding to a $\gd$ (and we see from the multivariate generating function $s$ that every word encoding simple permutations in this class contains at least one $\gd$) may be inflated by an arbitrary member of $\Av(231,3124)$, a class counted by the Fibonacci numbers of even index ($1$, $2$, $5$, $13$, $34$, $\dots$).  Therefore, the generating function we are interested in is given by
\[
f
=
(xc+x)f
+
\frac{x-2x^2+x^2}{1-3x+x^2}c
+
\frac{s(c,m,c,m)}{m}\frac{x-x^2}{1-3x+x^2}.
\]
As this equation is linear in $f$, it is trivial to obtain our final result.

\begin{theorem}
The generating function for $\Av(4231, 3124)$ is
\[
\frac{1-8x+20x^2-20x^3+10x^4-2x^5 - (1-4x+2x^2)\sqrt{1-4x}}
{2(1-3x+x^2)(-1+5x-4x^2+x^3)}.
\]
\end{theorem}

The first several terms of this sequence are
\[
1, 2, 6, 22, 88, 363, 1508, 6255, 25842, 106327, 435965, 1782733, 7275351, 29648647,
\]	
sequence \OEISlink{A165535} in the \OEISref. The radius of convergence is the smallest positive root of the cubic factor in the denominator, approximately 0.2451, and hence $f_n^{1/n} \to 4.0796\dots$

\section{Conclusion}\label{sec-infinite-conclusion}

It should be noted that the three examples presented in this paper are not the only $2\times 4$ classes which have been enumerated using these techniques.  In \cite{albert:counting-1324-4:}, the present authors used a precursor of this approach to enumerate $\Av(4231,1324)$; in the language of this paper, they proved that the simple permutations in this class are contained in
\[
\Geom\fnmatrix{rrr}{1&0&-1\\0&\bullet&0\\-1&0&1}.
\]
(Here the $\bullet$ entry denotes a cell in the standard figure filled with a unique point; this notion is formally defined in \cite[Section 10]{albert:geometric-grid-:}.)

Finally, we point out that there may be other $2\times 4$ classes to which these techniques apply.  While there is a decision procedure to determine whether a given class lies in a monotone grid class (see Huczynska and Vatter~\cite{huczynska:grid-classes-an:}), there is no known procedure to determine whether a given class lies in a \emph{geometric} grid class (and indeed, there are indications that this question may be quite difficult).  Needless to say, deciding whether the simple permutations of a given class lie in a geometric grid class is expected to be more difficult still.

\bigskip

\def\cprime{$'$}

\end{document}